\newif\ifsoda
\sodafalse

\ifsoda

\documentclass[twoside,leqno,twocolumn]{article}
\usepackage[hmargin=1in,vmargin=1in]{geometry}
\usepackage[nocopyright,siamtitle]{jeffeproc}

\usepackage{ltexpprt}
\else

\documentclass[11pt,twoside]{article}
\usepackage[hmargin=1in,vmargin=1in]{geometry}
\usepackage[charter]{mathdesign}

\usepackage{berasans,beramono}
\usepackage[mathcal]{eucal}
\SetMathAlphabet{\mathsf}{bold}{\encodingdefault}{\sfdefault}{b}{\updefault}
\SetMathAlphabet{\mathtt}{bold}{\encodingdefault}{\ttdefault}{b}{\updefault}
\SetMathAlphabet{\mathsf}{normal}{\encodingdefault}{\sfdefault}{\mddefault}{\updefault}
\SetMathAlphabet{\mathtt}{normal}{\encodingdefault}{\ttdefault}{\mddefault}{\updefault}

\fi

\usepackage{jeffe}
\ifsoda
\def\EMPH#1{\textcolor{BrickRed}{{\emph{#1}}}}
\fi

\usepackage[utf8]{inputenc}
\usepackage{graphicx,wrapfig}
\DeclareGraphicsExtensions{.pdf,.png,.jpg}
\graphicspath{{./}{./Fig/}}
\usepackage{microtype}
\usepackage{footmisc}
\usepackage[dvipsnames,usenames]{xcolor}
\usepackage[nocompress]{cite}
\usepackage{amsmath,mathtools,stmaryrd}
\usepackage{enumerate}

\usepackage{arydshln}
\DeclareMathOperator{\tort}{tort}
\DeclareMathOperator{\len}{len}

\def\Tangle{\Theta}
\def\bd{\partial\!}

\ifsoda
\else
\newtheorem{lemma}{Lemma}[section]
\newtheorem{corollary}[lemma]{Corollary}
\newtheorem{theorem}[lemma]{Theorem}
\numberwithin{figure}{section}
\fi

\ifsoda
\def\note#1{}
\fi

\begin{document}

\ifsoda
\else
\begin{titlepage}
\fi
\title{Tightening Curves on Surfaces Monotonically with Applications%
}

\author{
Hsien-Chih Chang\thanks{
Department of Computer Science, Duke University, USA.  Work by this author was partially supported by NSF under grants CCF-14-08763, CCF-15-13816, CCF-15-46392,
and IIS-14-08846, by an ARO grant W911NF-15-1-0408, and by BSF Grant 2012/229 from the
U.S.-Israel Binational Science Foundation.}
\and
Arnaud de Mesmay\thanks{Université Paris-Est, LIGM, CNRS, ENPC, ESIEE Paris, UPEM, Marne-la-Vallée, France. Work by this author was partially supported by the French ANR projects ANR-18-CE40-0004-01 (FOCAL), ANR-17-CE40-0033 (SoS), ANR-16-CE40-0009-01 (GATO) and ANR-19-CE40-0014 (MINMAX).}
}

\ifsoda
\else
\date{\today}
\fi

\maketitle

\begin{abstract}
We prove the first polynomial bound on the number of \emph{monotonic} homotopy moves required to tighten a collection of closed curves on any compact orientable surface, where the number of crossings in the curve is not allowed to increase at any time during the process.
The best known upper bound before was exponential, which can be obtained by combining the algorithm of de Graaf and Schrijver [\textit{J.\ Comb.\ Theory Ser.\ B}, 1997] together with an exponential upper bound on the number of possible surface maps.
To obtain the new upper bound we apply tools from hyperbolic geometry, as well as operations in graph drawing algorithms---the cluster and pipe expansions---to the study of curves on surfaces.

As corollaries, we present two efficient algorithms for curves and graphs on surfaces.
First, we provide a polynomial-time algorithm to convert any given multicurve on a surface into minimal position.  Such an algorithm only existed for single closed curves, and it is known that previous techniques do not generalize to the multicurve case.
Second, we provide a polynomial-time algorithm to reduce any $k$-terminal plane graph (and more generally, surface graph) using degree-1 reductions, series-parallel reductions, and $\Delta Y$-transformations for arbitrary integer $k$.
Previous algorithms only existed in the planar setting when $k \le 4$, and all of them rely on extensive case-by-case analysis based on different values of $k$.
Our algorithm makes use of the connection between electrical transformations and homotopy moves, and thus solves the problem in a unified fashion.


\end{abstract}
\ifsoda
\else
\setcounter{page}{0}
\thispagestyle{empty}
\end{titlepage}

\pagestyle{myheadings}
\markboth{Hsien-Chih Chang and Arnaud de Mesmay}
	{Tightening Curves on Surfaces Monotonically with Applications}
\fi

\section{Introduction}
\label{S:intro}

Let $\Sigma$ be an arbitrary compact orientable surface, possibly with boundary.
We consider a collection of closed curves (referred to as a \emph{multicurve}) on $\Sigma$ drawn in general position---with finitely many double crossings, each of which is a transverse intersection, and no tangents or crossings of higher orders.
The goal is to \emph{tighten} the closed curves into another collection of curves with a minimum number of crossings using only continuous deformations known as \emph{homotopy}.
The minimum number of crossings achievable under homotopy is known as the \emph{geometric intersection number}, a fundamental topological parameter associated with any set of closed curves on a surface.
There are many previous works, both in theory and in practice, describing how to compute the geometric intersection number and the tightening process more or less efficiently; we refer to the extensive historical notes of Despré and Lazarus~\cite{dl-cginc-2017} for more background on this classical problem.

Different papers measured the efficiency of the algorithms in different ways; in this paper we are in particular interested in minimizing the total number of combinatorial changes to the closed curves.
Using standard arguments~\cite{a-cas-26,ab-tkc-26,r-k-32},
every homotopy can be decomposed into a finite sequence of local changes called the \EMPH{homotopy moves}, consisting of the following three basic operations: undo a monogon, remove a bigon, and flip a triangle.
See Figure~\ref{F:homotopy} for an illustration.

\begin{figure*}[h]
  \centering
  \includegraphics[width=14cm]{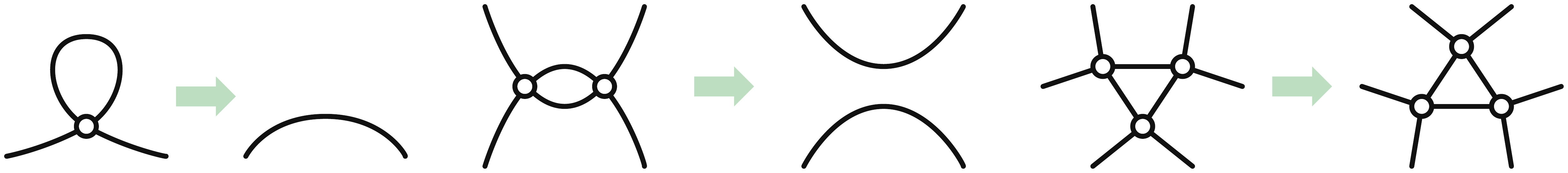}
  \caption{The three homotopy moves $\arc10$, $\arc20$, and $\arc33$.}
\label{F:homotopy}
\end{figure*}

Here our goal is to provide an upper bound on the number of homotopy moves used to tighten a given collection of curves.
Furthermore, a desired property of the tightening process is that at no times the number of crossings increases throughout the homotopy.
Intuitively this is a natural property to assume; after all, the goal is to minimize the final number of crossings, and in a sense we want to perform the algorithm greedily and never make the curves more complicated.
In addition, as we will explain later on, monotonicity is not merely a natural assumption to enforce on the tightening process, it is also the key property to draw connection to other sets of local transformations.
Surprisingly, proving that such a monotonic tightening process exists actually requires quite involved arguments, and it was only shown by Hass and Scott~\cite{hs-scs-94} and de Graaf and Schrijver~\cite{gs-mcmcr-97}
that we can safely make such an assumption.
Both algorithms used some discrete variants of the curve-shortening technique of Grayson~\cite{g-sec-89}, Shepard~\cite{s-tscs-91}, and Angenent~\cite{a-pecs2-91}.
The main downside of the approach is that these algorithms are not \emph{efficient} when measured in combinatorial changes.
Indeed, none of the authors of previous algorithms analyze their performance, and with careful reading the best upper bound on the number of homotopy moves performed is merely exponential.
%
Ideally, we would like to have the best of two worlds---a tightening process that is efficient while never creating new crossings.

\subsection{Our results}

In this paper we prove that any collection of closed curves on an orientable surface (possibly with boundary) of genus $g\ge 2$ can be tightened using a polynomial number of monotonic homotopy moves.

\begin{theorem}
\label{Th:monotone}
Any $n$-vertex multicurve $\gamma$ on an orientable surface $\Sigma$ of genus $g$ with $b > 0$ boundary components can be tightened monotonically using $O((g+b)n^3)$ homotopy moves.
When the surface~$\Sigma$ does not have any boundary component (that is, $b=0$) and not a torus, the upper bound becomes $O(n^5\log^3 g/g^2 + gn^3)$.
\end{theorem}

Note that our theorem applies to surface with any combination of genus and number of boundary components with the one exception of a boundaryless torus.
The result improves over the previous monotonic reduction algorithm by de Graaf and Schrijver~\cite{gs-mcmcr-97} which, combined with the exponential bound on the number of surface maps~\cite{bc-anrms-86} yielded an exponential upper bound.
A recent article by Chang~\etal~\cite{untangle} gave a polynomial upper bound on the number of homotopy moves.
However, their algorithm does not guarantee monotonicity, and it only works for a \emph{single} closed curve;
in fact, it is understood that a completely new approach is required to overcome these shortcomings.
Their algorithm relies on the \emph{bigon removal approach} powered by the result of Hass and Scott~\cite{hs-ics-85}, proving the existence of a \emph{singular bigon} or \emph{monogon}---that is, a bigon or monogon that overlaps itself in a not-too-pathological way.
Such bigons and monogons can be removed using polynomially many homotopy moves (see for example~\cite[\S4.2]{untangle} and~\cite[\S6.2.1]{changthesis}).
However, it is known (see for example Figure~0.1 of Hass and Scott~\cite{hs-ics-85} or Figure~1 of Despr\'e and Lazarus~\cite{dl-cginc-2017}) that such singular bigons may not exist when leaving the realm of single closed curves.
Furthermore, every known algorithm that removes singular bigons increases the number of crossings temporarily during the homotopy process, and therefore is not monotonic.

\medskip
The first application of our main theorem is that one can convert any given collection of closed curves into minimal position (that is, with a minimum number of crossings) using homotopy \emph{in polynomial time}.

\begin{theorem}
\label{Th:minimal}
Given a multicurve $\gamma$ on an orientable surface $\Sigma$, we can compute a minimal position of~$\gamma$ on~$\Sigma$ in polynomial time.
\end{theorem}

As a corollary, we can compute in polynomial time the geometric intersection number of a multicurve, a problem for which the first polynomial-time algorithm was only provided very recently by Despr\'e and Lazarus~\cite{dl-cginc-2017}.
In that paper, Despr\'e and Lazarus also provide a different algorithm to compute minimal position of a \emph{single} closed curve in polynomial time.
Since their techniques also rely on finding and removing singular bigons and monogons, it suffers from the same limitations as explained above and cannot be readily generalized to the more general setting of multicurves.
The existence of an efficient algorithm is not immediate even assuming our main theorem; one has to carefully examine each step of the proof and make sure they can be implemented efficiently.
While Theorem~\ref{Th:monotone} does not apply to the case where $\Sigma$ is a torus, we provide a separate algorithm to handle it.

\medskip
The second application to the main theorem is the first polynomial-time algorithm that reduces any $k$-terminal plane graph (and more generally, any $k$-terminal surface graph) using \emph{electrical transformations}---a collection of operations on surface graphs consists of degree-1 reductions, series-parallel reductions, and $\Delta Y$-transformations.
It is required that all transformations respect the embedding of the graph, and no terminals can be removed during the reduction.
The goal is to perform a sequence of electrical transformations on the input surface graph and reduce the graph as much as possible---that is, to obtain another surface graph that minimizes the number of edges.

\begin{theorem}
\label{Th:electrical}
Any surface graph with terminals can be reduced as much as possible using electrical transformations in polynomial time.
\end{theorem}

Electrical transformations have been widely applied to graph algorithms and network optimizations~\cite{k-etscn-1899,a-wtns-60,l-wtpn-63,f-erpns-85,ghp-igvsp-17}
and other fields of science and engineering~\cite{cpv-nastc-95,j-smtra-95,nt-aafts-96,st-afkrm-02,zg-effn-07,pss-alssa-18}.
For a history of electrical transformations and other related work, see~\cite{changthesis}.
%
%
The relation between electrical transformations and homotopy has been studied implicitly since Tait~\cite{t-k1-1876}, Steinitz~\cite{s-pr-1916,sr-vtp-34}, and Yajima and Kinoshita~\cite{yk-gk-57}, and explicitly by
Goldman and Kauffman~\cite{gk-kten-93} and Nobel and Welsh~\cite{nw-kg-00}, through the lens of \emph{medial construction}.
The \emph{medial graph $G^\times$} of a surface embedded graph $G$ is constructed as follows: create a vertex for each edge in $G$, and create an edge between two vertices if the corresponding two edges share both a vertex and a face in $G$.
From the construction it is immediate that every vertex in $G^\times$ has degree~$4$.
So one can decompose the medial graph into a collection of curves~$\gamma$ by making each vertex of $G^\times$ an intersection point between two constituent curves of $\gamma$.
Quantitative connection between the two sets of operations has been established first in the plane~\cite{tangle}, and later for general surfaces~\cite{changthesis,homoelectric}.
The most important observation we rely on is the following:
Any polynomial upper bound on the number of monotonic homotopy moves required to tighten the medial multicurve $G^\times$ turns into a polynomial upper bound on the number of electrical transformations required to reduce the surface graph~$G$.
Furthermore, the same statement holds when one replaces ``number of moves'' with ``running time''.

There are polynomial-time algorithms that reduce any surface graph with $2$-terminals~\cite{t-drpg-89,fp-dtert-93}, $3$-terminals~\cite{g-dtaa-91,p-snrtt-01,gs-tdrpg-11}, and $4$-terminals~\cite{acgp-frpwg-00,dm-ftpdw-15}.
As for arbitrary value of $k$, previous algorithms assume special positions of the terminals, say when all terminals lie on a single face of the plane graph~\cite{g-dtaa-91,cgv-rep-96}.
All these algorithms, especially the ones for constant number of terminals, rely on heavy case-by-case analysis to characterize what the reduced graphs look like (for example, the work of Archdeacon \etal~\cite{acgp-frpwg-00} and Demasi and Mohar~\cite{dm-ftpdw-15} for the $4$-terminal case in total span more than a hundred pages).
In contrast, our algorithm functions in a unified way by transforming the graph reduction problem into a curve tightening problem on a surface using a set of local operations similar to homotopy moves (see Section~\ref{SS:electrical}), and therefore avoids the above complications.
The electrical reduction algorithm relies crucially on the fact that the curve tightening process is efficient, monotone, and works for multicurves; this is why previous results~\cite{gs-mcmcr-97,untangle,dl-cginc-2017} cannot be used.
An important subtlety is that the aforementioned algorithms to reduce surface graphs with terminals also allow the use of one additional move called the \emph{terminal-leaf contraction}.
\ifsoda
\else We explain in Section~\ref{SS:terminal-leaf} how this additional move can also be integrated within our framework. \fi

Our efficient electrical reduction algorithm is the conclusion of a long sequence of works~\cite{tangle,untangle,changthesis,homoelectric} and our main philosophical contribution---curves and graphs on surfaces can be reduced efficiently when measured in combinatorial changes.

\subsection{Technical contribution}
The proof of Theorem~\ref{Th:monotone} can be viewed as an amalgamation of the curve shortening algorithm of de Graaf and Schrijver~\cite{gs-mcmcr-97},
the cluster and pipe expansion technique from graph drawings~\cite{cbpp-ecpg-09,weak,ft-cmpd-18},
and the crossing minimization algorithm for flat braids originated from Geck and Pfeiffer in the context of word problem over symmetric groups~\cite{gp-icha-93,gs-mcmcr-97}.
%
The first step relies on hyperbolic geometry, which is very relevant to our tightening problem for the following reasons: (1) any (multi)curve on a surface endowed with a hyperbolic metric is homotopic to a unique (multi)geodesic, and (2) a primitive (multi)geodesic is in minimal position.
Our proof follows this approach closely but the key challenge is to control the combinatorics of the curves as well as the length of the process.

Therefore, our first step is to endow $\Sigma$ with a hyperbolic metric, and to move the multicurve $\gamma$ to a neighborhood of the unique collection of geodesics of its homotopy class.
Unlike de Graaf and Schrijver~\cite{gs-mcmcr-97}, we cannot afford to move the multicurve all the way until it reaches a canonical braid-like form.
Instead, we execute the curve shortening algorithm frugally until the curves lie in the $\e$-neighborhoods of its geodesics, where $\e$ is chosen just small enough to ensure that these neighborhoods do not cover the entire surface.
Since we know that the curves can be tightened further while staying in the neighborhoods, at this point it is safe to put a puncture on the uncovered surface and reduce the problem to curves on (orientable) surfaces with boundary.

The second step relies on the new observation that, for a collection of curves $\gamma$ on surface with boundary, one can perform a quadratic number of homotopy moves and put the curves into a \emph{pipe system}---a regular neighborhood of some one-dimensional skeleton graph.
Then multicurve $\gamma$ along with the pipe system are modified gradually by the \emph{expansion operations}, in a way that after polynomial many steps, each constituent curve of $\gamma$ is combinatorially close to a power of some primitive curve, which then can be turned into a canonical form that looks like a flat braid.
After reaching the braid form we use the crossing minimization algorithm~\cite{gp-icha-93,gs-mcmcr-97} to make $\gamma$ tight.

We summarize the above steps in the following two lemmas.
Let $\gamma$ be a collection of curves on a surface $\Sigma$ of negative Euler characteristic, and let $\gamma_*$ be the unique (multi)geodesic of $\gamma$ on $\Sigma$.
We say the multicurve $\gamma$ is \EMPH{$\e$-close} to the geodesic $\gamma_*$ if the lift of~$\gamma$ in the universal cover lies in an $\e$-neighborhood of the lift of $\gamma_*$.

\begin{lemma}
\label{L:follow-strip}
Let $\gamma$ be an $n$-vertex non-contractible multicurve on a surface $\Sigma$ of genus $g \ge 2$ without boundary, and let $\gamma_*$ be the unique geodesic of $\gamma$ on~$\Sigma$.
One can endow $\Sigma$ with a hyperbolic metric so that the multicurve $\gamma$ can be made $\e$-close to $\gamma_*$ using $O(n^5\log^3 g/g^2)$ monotonic homotopy moves for some $\e = \Theta(g/(n\log g))$; furthermore, the $\e$-neighborhood of $\gamma_*$ does not cover the whole surface $\Sigma$.
\end{lemma}

\begin{lemma}
\label{L:tighten-strip}
Let $\gamma$ be an $n$-vertex multicurve with no contractible components on an orientable surface $\Sigma$ of genus $g$ with $b>0$ boundary components.
Then $\gamma$ can be tightened using $O((g+b)n^3)$ monotonic homotopy moves.
\end{lemma}


Theorem~\ref{Th:monotone} follows rather directly from Lemma~\ref{L:follow-strip} and Lemma~\ref{L:tighten-strip}; this is explained in Section~\ref{SS:proof}.
We prove Lemma~\ref{L:follow-strip} in Section~\ref{S:move-to-geodesic}, and Lemma~\ref{L:tighten-strip} in Section~\ref{S:tightening-boundary}.  Applications are discussed in Section~\ref{S:applications}.


\section{Preliminaries}

Familiarity with basic concepts regarding the topology and geometry of surfaces will greatly ease the reading. We recommend Stillwell~\cite{s-ctcgt-93} for a general combinatorial introduction to the topic, and the first chapter of Farb and Margalit~\cite{fm-pmcg-11} for the specific topic on curves, surfaces, and hyperbolic geometry.

\subsection{Curves on surfaces}
\label{SS:surfaces}
A \EMPH{surface} $\Sigma$ is a two-dimensional (topological) manifold, possibly with boundaries.
All the surfaces in this article are compact, connected, and orientable. The \EMPH{Euler characteristic} of a surface~$\Sigma$ is $2-2g-b$, where $g$ is the genus and $b$ the number of boundary components in $\Sigma$.
A \EMPH{closed curve} on a surface $\Sigma$ is a continuous map $\gamma\colon S^1 \rightarrow \Sigma$.
A \EMPH{multicurve} is a collection of closed curves, which form its \EMPH{constituent curves}.
An \EMPH{arc} or \EMPH{path} on a surface $\Sigma$ is a continuous map $\gamma\colon [0,1] \rightarrow \Sigma$ with endpoints on the boundary.
In general, we refer to either a collection of closed curve or arcs as \EMPH{curves}.
We only consider generic curves, that is, curves with only a finite number of self-intersections which are transverse double points.
A \EMPH{subpath} of a curve~$\gamma$ is the restriction of $\gamma$ to an interval. A curve is \EMPH{simple} if it is injective.
We will consider sometimes closed curves as graphs embedded on $\Sigma$ by treating their self-intersection points as vertices and the maximal subpaths between these vertices as edges.
A \EMPH{tangle} is a collection of boundary-to-boundary paths $\gamma_1, \ldots, \gamma_s$ in a closed topological disk, which (self-)intersect only pairwise, transversely, and away from the boundary. We call each individual path a \EMPH{strand} of the tangle.

A \EMPH{homotopy} between two closed curves $\gamma_1$ and $\gamma_2$ is a continuous deformation $h\colon S^1 \times [0,1] \rightarrow \Sigma$ such that $h(\cdot,0)=\gamma_1$ and $h(\cdot,1)=\gamma_2$.
This definition extends naturally to arcs and multicurves.
A closed curve is \EMPH{contractible} if it is homotopic to a point.
In this article, we take the convention that if at some point in a homotopy, a multicurve contains a contractible closed curve which has degenerated to a point, we can \textit{remove} this contractible closed curve from the multicurve.%
\footnote{Notice that this differs from de Graaf and Schrijver~\cite{gs-mcmcr-97}, which is why they require one more homotopy move than we do.}
As explained in the introduction, classical arguments show that any homotopy between two closed curves in general position can be decomposed into a sequence of the three homotopy moves pictured in Figure~\ref{F:homotopy}.
A multicurve is \EMPH{tightened}, or \EMPH{tight}, or is a \EMPH{tightening}, or is in \EMPH{minimal position} if it has the smallest possible number of intersections among all the multicurves within its homotopy class.
Sometimes, it will be useful to specify in which surface a homotopy or a tightening lies, for example for a multicurve $\gamma$ contained in a surface $\Sigma$ which is a sub-surface of another surface $\Sigma'$; in this case, we will talk about a homotopy, or a tightening, \EMPH{within} $\Sigma$.

A \EMPH{monogon}%
\footnote{\label{fn:gons}In this work, we only care about \emph{embedded} monogons, bigons or trigons and thus only define those.  We refer to Hass and Scott~\cite{hs-ics-85} for an extensive review of other kinds of monogons and bigons and the corresponding existence results.}
for a curve $\gamma$ is a subpath that begins and ends at the same vertex $x$ and bounds a disk incident to only that vertex.
A \EMPH{bigon}\footref{fn:gons} for a curve~$\gamma$ consists of two simple interior-disjoint subpaths of $\gamma$, sharing two endpoints that together bound a disk on~$\Sigma$ incident to only these two endpoints.
Similarly, a \EMPH{trigon}\footref{fn:gons} for $\gamma$ consists of three simple interior-disjoint subpaths of $\gamma$, forming three pairwise intersections that together bound a disk on $\Sigma$.
A monogon, bigon, or trigon is \EMPH{empty} when the interior of the bounded disk is disjoint from~$\gamma$.
A bigon is \EMPH{minimal} or \EMPH{innermost} if the disk it bounds does not contain a smaller bigon or monogon; a \emph{minimal monogon} is defined similarly.
Note that a minimal monogon does not contain anything in its interior, since any strand crossing it will form an inner monogon or bigon.
Therefore, a minimal monogon can be removed by a single $\arc{1}{0}$ move.

An argument dating back to Steinitz~\cite{s-pr-1916,sr-vtp-34}
\ifsoda
\else
(see Hass and Scott~\cite[Lemma~1.4]{hs-scs-94})
\fi
shows that a minimal bigon can also be removed using $O(n)$ monotonic moves, where $n$ is the sum of the number of strands and interior vertices in the bigon.
%

\ifsoda
\else
We include the proof to be comprehensive.
\fi
\begin{lemma}\label{L:steinitz-empty}
A non-empty minimal bigon $\beta$ must have an empty trigon incident to one of the bounding curves.
Thus one can first remove all the $n$ vertices inside $\beta$ using $n$ $\arc33$ moves, followed by removing all $s$ strands of $\beta$ using $s$ $\arc33$ moves.
\end{lemma}

\begin{proof}
Let $\Tangle$ the tangle formed by $\gamma$ inside the bigon.
Each strand of $\Tangle$ is simple, otherwise it would form a monogon, and each pair of strands intersects at most once, otherwise they would form a bigon. Similarly, each strand intersects two distinct bounding curves. If there are no vertices in this tangle, there is an empty trigon formed by a vertex and one of the strands.

Otherwise, for every vertex $x$ of the tangle obtained by intersecting two strands $\alpha$ and $\beta$, the two strands $\alpha$ and $\beta$ both intersect one of the bounding curves $\lambda$, and thus define a trigon $R_x$ with it.
We denote the other two endpoints by $a$ and $b$, and look at such a vertex $x$ such that the trigon it defines is inclusion-wise minimal and one of its three endpoints is on $\lambda$.
Without loss of generality, $a$ is on $\lambda$, and no strand crosses $\alpha$ between $a$ and $x$.
If a strand crosses $\beta$ between $b$ and $x$, denote by $y$ the crossing point closest to $x$.
This strand does not cross $\alpha$ between $a$ and $x$, thus $R_y$ is a trigon inside $R_x$ and one of its endpoints is on $\lambda$, which contradicts minimality of $R_x$. Thus $R_x$ is empty.

We can recursively remove the vertices of the tangle $\Tangle$ using this empty trigon, using $n$ $\arc33$ moves. Then, using $s$ $\arc33$ moves we can remove all the strands, making the bigon empty.
\end{proof}

This allows us to remove minimal bigons using one last $\arc20$ move.
For convenience, we state the result independently as a lemma, and refer to as the \EMPH{Steinitz bigon removal algorithm}:
\begin{lemma}
\label{L:steinitz}
Any minimal bigon or monogon with $n$ interior vertices and $s$ strands can be removed using $n+s+1$ monotonic homotopy moves.
\end{lemma}

\subsection{Cut graphs and systems of arcs}
\label{SS:cut-graph}

A \EMPH{cellular embedding} of a graph $G$ on a surface $\Sigma$ is an injective map from $G$ to $\Sigma$ where all the faces (connected components of complement of the embedding) are homeomorphic to open disks.
A \EMPH{tree-cotree decomposition} of a cellularly embedded graph $G$ is a partition $(T,L,C)$ of the edges of~$G$ into three disjoint subsets:
a spanning tree $T$ of $G$, the edges $C$ corresponding to a spanning tree of the dual graph~$G^*$, and exactly $2g$ leftover edges $L \coloneqq E(G) \setminus (T\cup C)$, where $g$ is the genus of the underlying surface~\cite{e-dgteg-03}.
Let $\gamma$ be a multicurve on $\Sigma$; we temporarily view $\gamma$ as a 4-regular graph with some given embedding.  However, the embedding of $\gamma$ is not necessarily cellular; let $G$ be a cellular refinement of $\gamma$ obtained by triangulating every face.
A \EMPH{dual reduced cut graph $X$}~\cite{eh-ocsd-04} (hereafter, just \EMPH{cut graph}) is a cellularly embedded graph
%
%
obtained from a tree-cotree decomposition $(T,L,C)$ of $G$ as follows:  Start with the subgraph of $G^*$ containing the dual spanning tree $C^*$ and the leftover edges~$L^*$, repeatedly delete degree-one vertices, and finally perform series reductions on all vertices with degree two.

The cut graph $X$ inherits a cellular embedding into $\Sigma$ from the embedding of $G^*$; by construction, this embedding has exactly one face.  Because every vertex of $X$ has degree $3$, Euler's formula implies that $X$ has exactly $4g-2$ vertices and $6g-3$ edges.
We call the edges of~$X$ \EMPH{arcs}.  Cutting the surface $\Sigma$ along $X$ yields a polygon with $12g-6$ sides, which we call the \EMPH{fundamental polygon} of $X$.
The cut graph induces a regular tiling $\hat{X}$ of the universal cover~$\hat\Sigma$ of~$\Sigma$; we refer to each lift of the fundamental polygon of $X$ as a \EMPH{tile}.
By construction, the cut graph~$X$ satisfies the following \EMPH{crossing property}: \emph{Each edge of the curve~$\gamma$ crosses~$X$ at most once}.
%
%

\medskip
When $\Sigma$ is a surface with boundary, it can be cut into a planar piece using exclusively boundary-to-boundary paths: a \EMPH{system of arcs $\Xi$} is a collection of simple boundary-to-boundary paths that cuts the surface $\Sigma$ open into a single polygon.
Furthermore, for any closed curve $\gamma$ on $\Sigma$, there exists a system of arcs~$\Xi$ satisfying the following \EMPH{crossing property}: \emph{Each arc in $\Xi$ intersects each edge of $\gamma$ at most twice, and only transversely.}
We summarize this in the following lemma, and refer for example to Colin de Verdière and Erickson~\cite[Section~6.1]{ce-tnspc-10} or Erickson and Nayyeri~\cite[Section~3]{en-mcsns-11} for a proof of this, as well as polynomial-time algorithms to compute such $\Xi$.

\begin{lemma}
\label{L:system-of-arcs}
Let $\Sigma$ be an arbitrary genus-$g$ surface $\Sigma$ with $b$ boundary components.
There is a system of arcs $\Xi$ on $\Sigma$ of size $O(g+b)$ in general position relative to multicurve $\gamma$ such that each arc intersects each edge of $\gamma$ at most twice (and therefore every edge intersects~$\Xi$ at most $O(g+b)$ times).
Furthermore, $\Xi$ can be computed in $O(n\log n + (g+b)n)$ time, where $n$ is the number of crossings in~$\gamma$.
\end{lemma}

\subsection{Hyperbolic trigonometry}

We assume the readers have some familiarity with hyperbolic geometry.  While we recall most of the properties that we rely on, the hyperbolic intuition is sometimes significantly different from the Euclidean one.
We recommend Traver~\cite{t-thp-14} for a nice introduction to hyperbolic trigonometry.

Any surface of negative Euler characteristic can be endowed with a hyperbolic metric.
The \EMPH{area $A(\Sigma)$} of surface $\Sigma$ endowed by this metric is constrained related to the Euler characteristic by the Gauss-Bonnet formula: $A(\Sigma) = -2 \pi \chi(\Sigma)$.
The key hyperbolic property that we will rely on is that any closed curve is homotopic to a unique geodesic under a given hyperbolic metric.
When applying this property to a multicurve, we will refer to the collection of geodesics it yields as a \EMPH{multigeodesic}, or sometimes when there is no ambiguity, simply as a \EMPH{geodesic}.

The \EMPH{hyperbolic law of cosines} states that for a geodesic triangle with angles $\alpha, \beta$, and $\gamma$ and side lengths $a$, $b$, and $c$ (such that the segment of length $a$ is opposed to the angle $\alpha$, and similarly for the other ones), we have:
\[
\cos \alpha = -\cos \beta \cdot \cos \gamma + \sin \beta \cdot \sin \gamma \cdot \cosh a.
\]

A \EMPH{Saccheri quadrilateral} is a hyperbolic geodesic quadrilateral with two equal sides---called the \emph{legs}---perpendicular to a third side, called the \emph{base};
the fourth side is called the \emph{top}.
We denote the lengths of the legs, base, and top as $a$, $b$ and $c$, respectively.
Any Saccheri quadrilateral satisfies the following property:
\[
\sinh \frac{c}{2} = \cosh a \cdot \sinh \frac{b}{2}.
\]

A \EMPH{Lambert quadrilateral} is a hyperbolic geodesic quadrilateral with three right angles. Denoting by $\alpha$ the fourth angle, and by $a$ and $b$ the lengths of the two sides opposite to it, we have the formula
\[
\cos \alpha = \sinh a \cdot \sinh b.
\]

\subsection{Proving the main theorem}
\label{SS:proof}

We conclude the preliminaries by explaining how to prove Theorem~\ref{Th:monotone} assuming Lemmas~\ref{L:follow-strip} and~\ref{L:tighten-strip}:

\begin{proof}[of Theorem~\ref{Th:monotone}]
A theorem of Hass and Scott~\cite[Theorem~2.7]{hs-ics-85} shows that any non-simple \emph{contractible} closed curve on an orientable surface (or more generally, any closed curve homotopic to a simple curve) has an embedded monogon or bigon.
Therefore, any multicurve in which a closed curve is contractible and non-simple also contains an embedded bigon or monogon.
Note that a simple contractible closed curve crossing other components of the multicurve also forms embedded bigons.
Thus after removing all the embedded bigons and monogons using Lemma~\ref{L:steinitz}---which takes $O(n^2)$ moves for an $n$-vertex multicurve---we can assume that contractible components, if there are any, have been shrunk to points and removed.
Therefore throughout the article, we will assume that there are no contractible components in the multicurves considered.
Since any closed curve on a sphere or a disk is contractible, we directly obtain Theorem~\ref{Th:monotone} with an $O(n^2)$ bound in such cases.

Lemma~\ref{L:follow-strip} allows us to reduce the case of boundaryless surfaces (except the torus) to the case of surfaces with boundary.
Indeed, once a multicurve has been placed in an $\e$-neighborhood of its multigeodesic and the neighborhood does not cover the whole surface $\Sigma$, one can safely add a puncture (say an arbitrarily small boundary) outside this $\e$-neighborhood as it has no impact on the tightening of $\gamma$.
Surfaces with boundary are then dealt with using Lemma~\ref{L:tighten-strip}.
The case when $\Sigma$ is a boundaryless torus is not handled by Lemma~\ref{L:follow-strip} nor by the previous observations, and thus remains untackled in Theorem~\ref{Th:monotone}.
\end{proof}


\section{Moving Curves Close to Geodesics}
\label{S:move-to-geodesic}

In this section we prove Lemma~\ref{L:follow-strip}.
Let $\Tangle$ be a tangle whose disk is endowed with a Riemannian (say hyperbolic or Euclidean) metric so that it is strictly convex.
A tangle $\Tangle$ is \EMPH{straightened} if all the strands of $\Tangle$ are shortest paths with respect to the metric.
We emphasize the difference between \emph{straightened} and \emph{tightened}:
Tightening is a combinatorial condition where all strands are intersecting minimally;
straightening is a geometric condition where all strands are shortest paths.
A straightened tangle must be tightened.
A converse statement is provided by
the result of Shepard~\cite{s-tscs-91} and Neumann-Coto~\cite{n-csgs-01}, which says that any multicurve in minimal position on a surface $\Sigma$ must be shortest paths with respect to some metric on $\Sigma$ (but not necessarily hyperbolic~\cite{hs-ccgs-1999}).

We will make use of the following quantitative version of
Ringel's homotopy theorem~\cite{r-tedgo-56,r-ugal-57} (see also~\cite{s-otuok-92,hs-scs-94,gs-mcmcr-97,gf-pa-17}).
Since none of the earlier results proved the quadratic upper bound, we include a proof below to be comprehensive.

\begin{lemma}[Hass and Scott~{\cite[Lemma~1.6]{hs-scs-94}}]
\label{L:steinitz-ringel}
Any $m$-vertex tangle $\Tangle$ can be
straightened (with respect to some given metric) monotonically using $O(m^2)$ homotopy moves.
\end{lemma}

\begin{proof}
We denote by $\alpha_1, \ldots, \alpha_k$ the strands of a tangle $\Tangle$ in a disk $D$, and by $\delta_1, \ldots, \delta_k$ the shortest paths between their endpoints.
We will use several times the result of Steinitz's mentioned in Lemma~\ref{L:steinitz} in the preliminaries, that any innermost embedded bigon or monogon can be removed using a linear number of monotonic moves.

As a first step, we apply Lemma~\ref{L:steinitz} iteratively to any innermost bigon or monogon in the tangle~$\Tangle$.
Since removing a bigon or monogon reduces the number of vertices by at least one, this can be done in $O(m^2)$ moves, after which there are no embedded bigons nor monogons in $\Tangle$ anymore.
In particular, every strand $\alpha_i$ is simple and any pair of strands $\alpha_i$ and $\alpha_j$ crosses at most once (at this point, the tangle $\Tangle$ is tightened but not \emph{straightened}).

\begin{figure*}
   \centering
   \includegraphics[width=\textwidth]{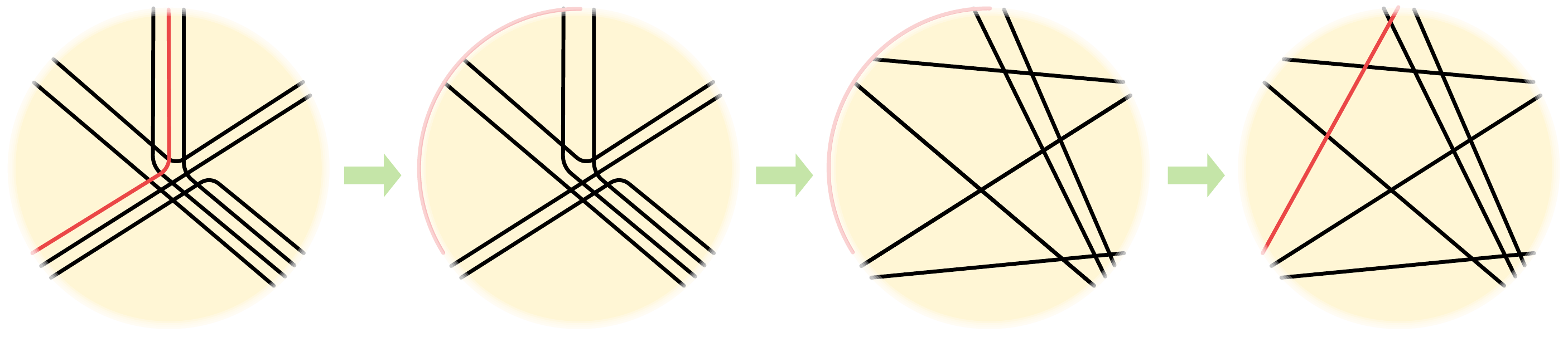}
   \caption{The inductive step to straighten a tangle in a Euclidean disk.}
   \label{F:ringel}
\end{figure*}

The remainder of the proof uses induction on the number of strands in the tangle.
The base case is trivial: a tangle made of a single strand can be straightened without using any move.
Inductive step is pictured in Figure~\ref{F:ringel}.

For an $s$-strand tangle $\Tangle$, we consider the bigons formed by one subpath of a strand and a subpath of the disk boundary $\bd D$.
Since all strands are simple and all bigons between any two strands of $\Tangle$ were removed, we can find a bigon between some $\alpha_i$ and $\bd D$ that is innermost.
Such a bigon is only crossed transversely by other strands of the tangles and its vertices can be removed using $O(m)$ $\arc33$ moves by applying the first step of Lemma~\ref{L:steinitz-empty}.
Once the bigon contains no vertices, we can move $\alpha_i$ towards $\bd D$ until $\alpha_i$ is arbitrarily close to $\bd D$.
One can then consider a slightly smaller disk than $D$ that contains all the strands of the tangle $\Tangle$ except $\alpha_i$.
The new tangle defined by the smaller disk is then straightened recursively.
Then, what remains is to move $\alpha_i$ to the shortest path $\delta_i$.
Since shortest paths cross minimally, they do not form bigons.
And because $\alpha_i$ was chosen so that the bigon formed with a subpath of $\bd D$ was innermost, as all the other strands have been straightened inductively,
the bigon between $\alpha_i$ and $\delta_i$ must be innermost, and can be swept using $O(m)$ moves again by Lemma~\ref{L:steinitz}.
Note that if $\alpha_i$ does not cross any other strand of $\Tangle$, this inductive step costs zero moves.
The total number of moves used throughout the recursion is therefore $O(s' \cdot m)$, where $s'$ is the number of strands crossing at least some other strand.
By charging each of these strands to one of their crossing points, we have $s'=O(m)$, and therefore the total bound on number of moves is $O(m^2)$.
\end{proof}

We will use the following corollary of Lemma~\ref{L:steinitz-ringel} in subsequent sections.

\begin{corollary}
\label{C:empty-trigon}
Let $\delta$ be a trigon with $m$ vertices.
Trigon $\delta$ can be made empty using $O(m^2)$ monotonic homotopy moves in a small neighborhood of $\delta$.
\end{corollary}

\begin{proof}
Consider the tangle $\Tangle$ formed by taking a small neighborhood of the trigon $\delta$.
We endow $\Tangle$ with a metric in such a way that the trigon, formed by replacing the three strands that are the bounding curves of $\delta$ with shortest paths with respect to the metric, has its orientation opposite to that of $\delta$.
To see the existence of such metric, first we endow $\Tangle$ with a metric of constant curvature (say a hyperbolic metric).
Notice that by deforming the distances within a small neighborhood of the disk boundary we can realize arbitrary spacings between endpoints of the strands.
Now by placing the endpoints of the three bounding curves of $\delta$ carefully and connecting each pair of them using shortest paths (with respect to the endowed metric of constant curvature), one can realize either orientation of the trigon.
Applying Lemma~\ref{L:steinitz-ringel} to $\Tangle$ with respect to the constructed metric empties and flips the trigon $\delta$ in $O(m^2)$ moves; we terminate the algorithm just before $\delta$ is flipped.
\end{proof}

\subsection{Constructing the hyperbolic metric}

In this subsection, we explain how to endow $\Sigma$ with a hyperbolic metric that is well-tailored to the purpose of tightening $\gamma$.

\begin{lemma}
\label{L:hyperbolic-metric}
Let $\Sigma$ be a boundaryless surface of genus $g\geq 2$ and $\gamma$ be an $n$-vertex non-contractible multicurve on $\Sigma$.
There is a \EMPH{hyperbolic metric $d_H$} on $\Sigma$ such that
\begin{enumerate}[(1)]\itemsep=0pt
  \item multicurve $\gamma$ can be turned into another multicurve $\gamma'$ of length $O(n\log g)$ using $O(n^2)$ monotonic homotopy moves, and
  \item the length of the shortest non-contractible cycle on $\Sigma$ (known as the \emph{systole}) is at least $1$.
\end{enumerate}
\end{lemma}
%

\begin{proof}
The construction is similar to the argument in Dehn's seminal result \cite{d-uudg-11} that the graph distance on a regular tiling of the universal cover $\hat\Sigma$ approximates the hyperbolic metric on $\hat\Sigma$.
Construct a cut graph $X$ from the curve $\gamma$ such that every edge of $\gamma$ crosses $X$ at most $O(1)$ times, as described in Section~\ref{SS:cut-graph}.
Lift the cut graph $X$ to the universal cover endowed with the unique hyperbolic metric, such that the edges of $X$ are geodesic segments of equal length and each corner has angle $1/3$ circles; this implies, using the hyperbolic law of cosines, that each side of the fundamental polygon has length at least~$1$.%
\footnote{To be accurate, the side length is equal to 
$2\cosh^{-1}\Paren{ \sin(2\pi/6) \cdot \cos(2\pi / (24g-12)) }$ which is bigger than~$1$ for all $g \ge 2$.}
Note that the diameter of the fundamental polygon is $O(\log g)$, which also follows from the hyperbolic law of cosines. One can project the metric back to the original surface; denote the hyperbolic metric constructed as~{$d_H$}.


To prove that the hyperbolic metric $d_H$ defined on surface $\Sigma$ satisfies item~(1), consider the modified curve~$\gamma'$ where all strands within the open disk $\Sigma\setminus X$ are straightened using Lemma~\ref{L:steinitz-ringel}.
As per lemma, $\gamma'$ can be obtained from $\gamma$ using $O(n^2)$ moves.
Note that any geodesic path not intersecting $X$ has length at most the diameter of the fundamental polygon with respect to $d_H$, which is $O(\log g)$. 
This directly implies that the length of $\gamma'$ is at most $O(n\log g)$, thus the hyperbolic metric $d_H$ satisfies item~(1).

As for item~(2), consider any non-contractible cycle $\sigma$ on surface $\Sigma$; without loss of generality assume $\sigma$ to be a geodesic.
If we lift $\sigma$ to the universal cover $\hat\Sigma$ such that the lift $\hat\sigma$ starts and ends on the lift $\hat X$ of the cut graph $X$, because $\sigma$ is non-contractible, the two arcs of $\hat X$ where $\hat\sigma$ starts and ends respectively are two different translates of the same arc in $X$.
Consider the sequence of arcs $a_0,\dots,a_k$ in $\hat X$ intersected by $\hat\sigma$.  Because $\sigma$ is a geodesic and every vertex in $\hat X$ has degree $3$, one has $a_i \neq a_{i+1}$ and no $a_i$ is incident to $a_{i+2}$ for all $i$.
If for some $i$ the two arcs $a_i$ and $a_{i+1}$ are not incident to each other (that is, $a_i$ and $a_{i+1}$ do not share a vertex in $\hat X$), then by hyperbolic trigonometry the length of the subpath of $\hat\sigma$ connecting $a_i$ to $a_{i+1}$ is at least the length of the side of the polygon, which is at least $1$.
Otherwise, if $a_{i}$ is incident to $a_{i+1}$ and $a_{i+1}$ is incident to $a_{i+2}$, as $a_{i}$ is not incident to $a_{i+2}$, by reflecting the subpath of $\hat\sigma$ from $a_{i+1}$ to $a_{i+2}$ to the tile that contains $a_{i}$ and $a_{i+1}$ we again have the length of the subpath from $a_i$ to $a_{i+1}$ of $\hat\sigma$ lower-bounded by the length of $a_{i+1}$.
This proves that $d_H$ satisfies item~(2).
\end{proof}

\subsection{Straightening multicurve using disks}

\paragraph{Tortuosity.}
Let $\gamma$ be a multicurve on $\Sigma$.  Denote \EMPH{$D(x,r)$} the disk centered at point $x$ with radius~$r$ (with respect to the constructed metric $d_H$ in Lemma~\ref{L:hyperbolic-metric}).
Denote the two endpoints of the maximal subpath of $\gamma$ in $D(\gamma(t),1/2)$ containing $\gamma(t)$ as $x$ and~$y$, and the maximal subpath itself as $\gamma[x,y]$.
The \EMPH{tortuosity}~\cite{gs-mcmcr-97} of the multicurve $\gamma$ at point $t$, denoted as \EMPH{$\tort(\gamma,t)$}, is the difference between the length of the subpath of $\gamma$ lying in $D(\gamma(t),1/2)$ and the geodesic distance between the two endpoints of the subpath:
\[
\tort(\gamma,t) \coloneqq \len\Paren{\gamma[x,y]} - d_H(x,y).
\]
In practice, the tortuosity of $\gamma$ at point $t$ lower bounds the improvement one will make after straightening the disk $D(\gamma(t), 1/2)$.
The {tortuosity} of a multicurve $\gamma$ is the supremum of $\tort(\gamma,t)$ where $t$ ranges over $[0,1]$.
The goal of the following lemma is to prove that when the tortuosity of a multicurve is small, then the whole multicurve is $\e$-close to its multigeodesic.
In other words, as long as the multicurve $\gamma$ has points that are at least $\e$ away from the geodesic, we can always find a disk centered at some point of $\gamma$ whose straightening will decrease the length of $\gamma$ by at least fixed amount, depending only on $\e$.

\begin{lemma}
\label{L:tortuosity}
For any $\e > 0$ smaller than the systole of $\Sigma$, if the tortuosity of $\gamma$ is at most $O(\e^2)$, then $\gamma$ is $\e$-close to the multigeodesic~$\gamma_*$.
\end{lemma}

\begin{proof}
We will prove the contrapositive statement using hyperbolic trigonometry.  For the sake of generality we temporarily treat $r$ as a variable; at the end of the calculation one just plugs in $r\coloneqq 1/2$.
Here we list two identities that will be used in our proof.
\begin{enumerate}[(1)]
\item For any real number $x$, $\sinh(2x) = 2\sinh x \cosh x$ and $(\cosh(x))^2 - (\sinh(x))^2 = 1$.


\item Given an arbitrary \emph{Saccheri quadrilateral} with the lengths of the legs, base, and top as $a$, $b$, and $c$ respectively, then
\[
\sinh\frac{c}{2} = \cosh a \cdot \sinh\frac{b}{2}.
\]
\end{enumerate}

\begin{figure*}
  \centering
  \def\svgwidth{0.9\textwidth}
  \scriptsize
  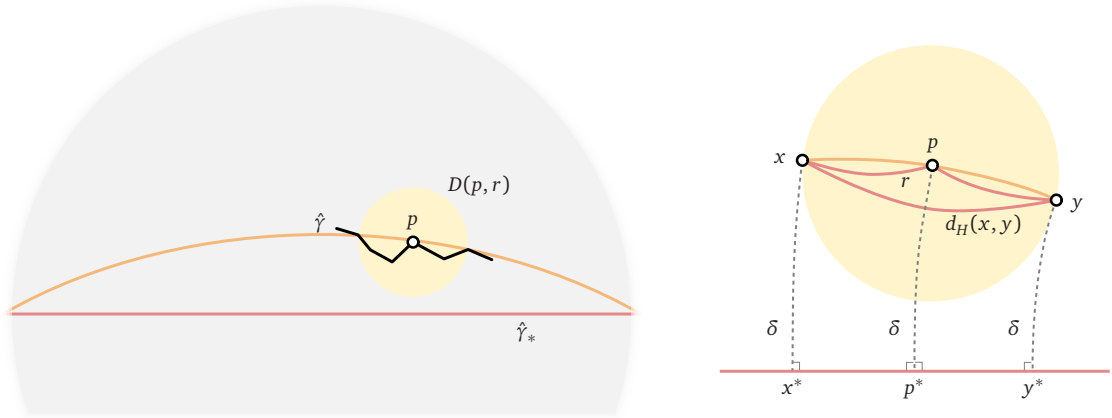
  \caption{Left: The setup for proof of Lemma~\ref{L:tortuosity}, represented in the Poincaré disk model.
  The orange curve is the set of points at distance exactly $\delta$ from $\hat\gamma_*$.
  This \textit{hypercircle} is not a geodesic, but is always a circular arc in the Poincaré disk.
  Right: Zooming around the disk $D(p,r)$. The tortuosity is minimized when $x$ and $y$ lie on the orange hypercircle. The geodesics between $x$, $p$, and $y$ are in red, and three Saccheri quadrilaterals are formed by red and dotted geodesic arcs.}
  \label{F:hyperbolic}
\end{figure*}

Lift both $\gamma$ and $\gamma_*$ to the universal cover $\hat\Sigma$; denote the resulting families of paths as $\hat\gamma$ and $\hat\gamma_*$ accordingly.
Let $t$ be a point in $[0,1]$ such that $\hat\gamma(t)$ has maximum distance to $\hat\gamma_*$.  Refer to point $\hat\gamma(t)$ as $p$ and the maximum distance as $\delta$; by assumption $\delta$ is at least $\e$.
Our goal is to prove that the tortuosity of $\gamma$ at $t$ is at least $\Omega(\e^2)$.
One has
\[
\tort(\gamma,t) = \len\Paren{\hat\gamma[x,y]} - d_H(x,y) \ge 2r - d_H(x,y).
\]
Here without loss of generality we will assume that $x$ and $y$ are both at distance exactly $\delta$ to $\hat\gamma_*$.
The reason one can make such an assumption is that, as one moves $x$ and $y$ perpendicularly along the geodesics away from~$\hat\gamma_*$, $d_H(x,y)$ increases and therefore the tortuosity when both $x$ and $y$ are at distance $\delta$ is a lower bound to the original tortuosity. See Figure~\ref{F:hyperbolic}.

What is left is to upper bound $d_H(x,y)$.
Let $x^*$, $p^*$, and $y^*$ be the points on $\hat\gamma_*$ that have minimum distance to $x$, $p$, and $y$ respectively.  By identity (2) one has
\[
\sinh \Paren{d_H(x,y)/2} = \cosh\delta \cdot \sinh \Paren{d_H(x^*,y^*)/2}
\]
and
\[
\sinh (r/2) = \cosh\delta \cdot \sinh \Paren{d_H(x^*,y^*)/4}.
\]
The second equality gives us
\[
{d_H(x^*,y^*)/2} = 2\sinh^{-1} \Paren{\frac{\sinh(r/2)}{\cosh\delta}},
\]
which we plug back in the first equation to get
\[
\sinh \Paren{d_H(x,y)/2} = \cosh\delta \cdot \sinh \Paren{2\sinh^{-1} \Paren{\frac{\sinh(r/2)}{\cosh\delta}}}.
\]
Apply identity (1) on the first hyperbolic sine, one has
\ifsoda
\else
\begin{align*}
\sinh \Paren{d_H(x,y)/2} &= \cosh\delta \cdot 2 \cdot \sinh \Paren{\sinh^{-1} \Paren{\frac{\sinh(r/2)}{\cosh\delta}}} \cdot \cosh \Paren{\sinh^{-1} \Paren{\frac{\sinh(r/2)}{\cosh\delta}}} \\
&= \cosh\delta \cdot 2 \cdot \Paren{\frac{\sinh(r/2)}{\cosh\delta}} \cdot \cosh \Paren{\sinh^{-1} \Paren{\frac{\sinh(r/2)}{\cosh\delta}}} \\
&= 2 \cdot {\sinh(r/2)} \cdot \Paren{1+ \Paren{ \sinh\Paren{\sinh^{-1} \Paren{\frac{\sinh(r/2)}{\cosh\delta}}} }^2 }^{1/2} \\
&= 2 \cdot {\sinh(r/2)} \cdot \Paren{1+ \Paren{ \frac{\sinh(r/2)}{\cosh\delta} }^2 }^{1/2}.
\end{align*}
This shows that
\begin{align*}
{d_H(x,y)} &= 2 \cdot \sinh^{-1} \Paren{ 2 \cdot {\sinh(r/2)} \cdot \Paren{1+ \Paren{ \frac{\sinh(r/2)}{\cosh\delta} }^2 }^{1/2} }.
\end{align*}
\fi
%
Taylor expand $d_H(x,y)$ around $\delta=0$ gives us
\[
d_H(x,y) = 2r - \frac{(\sinh(r/2))^3}{\cosh(r/2)\cdot\cosh(r)}\delta^2 + O(\delta^4),
\]
and therefore $\tort(\gamma,t) \ge \Omega(\delta^2) \ge \Omega (\e^2)$.
\end{proof}

Let us emphasize here how resolutely hyperbolic this lemma is. It works because a line equidistant to a geodesic (here $\hat\gamma_*$) is \textit{not} a geodesic in hyperbolic geometry, and it is this defect of geodesicity that we exploit to lower bound the tortuosity.
Comparatively, in Euclidean geometry, a line equidistant to a straight line is again a straight line, and thus there is no analogue lemma.
This is why our proof techniques do not apply to the boundaryless torus.

\paragraph{Exposing points outside the neighborhood.}
Now we proceed to upper bound $\e$ so that the $\e$-neighborhood of the multigeodesic $\gamma_*$ does not cover the whole surface $\Sigma$.

\begin{lemma}
\label{L:strip-width}
Let $\gamma$ be an $n$-vertex multicurve on $\Sigma$.  Then the $\e$-neighborhood of $\gamma_*$ does not cover the whole surface~$\Sigma$ if $\e$ is at most $O(g/(n \log g))$.
\end{lemma}

\begin{proof}
Given any multicurve $\gamma$ with the corresponding multigeodesic $\gamma_*$ on the surface $\Sigma$ with the constructed hyperbolic metric $d_H$, the length of $\gamma_*$ is at most $O(n\log g)$ by Lemma~\ref{L:hyperbolic-metric}(1).
For small enough~$\e$, the area of the $\e$-neighborhood of a multicurve with length $\ell$ is at most $O(\e\ell)$.
To see this, cover the neighborhood with kite-like \emph{Lambert quadrilaterals} with length of the short sides as $\e$.  The only acute angle $\alpha$ of the quadrilateral is equal to $\arccos((\sinh \e)^2)$.  The area of the quadrilateral is equal to the angle deficit, which is $\pi/2 - \alpha$.
Therefore the area of the quadrilateral is at most $O(\e^2)$, and thus the total area of the $\e$-neighborhood on $\Sigma$ is at most $O(\e^2 \cdot \ell/\e) = O(\e\ell)$.

The area of the surface is precisely $(4g-4)\pi$.  (This follows directly from the Gauss-Bonnet theorem which is independent of the hyperbolic metric up to scaling.%
\footnote{Alternatively, one can derive the area directly: divide the fundamental polygon into $12g-6$ triangles by drawing straight-lines from the center of the polygon to all vertices, and use the area formula for triangles.})
This implies that for the $\e$-neighborhood of $\gamma_*$ to cover the whole surface $\Sigma$, the following holds:
\[
\e ~\ge~ \frac{(4g-4)\pi}{O(n\log g)} ~\ge~ \Omega\Paren{\frac{g}{n \log g}}.
\]
In other words, if we set $\e \le O(g/(n \log g))$, then the $\e$-neighborhood of $\gamma_*$ cannot cover the whole surface $\Sigma$, thus proving the lemma.
\end{proof}

\ifsoda
\else
Basmajian, Parlier, and Souto~\cite{basmajian2017geometric} showed that for any fixed genus $g$, the $O(1/n)$ bound in Lemma~\ref{L:strip-width} is tight up to logarithmic factors.
\fi

\subsection{Putting it together}
Now we are ready to prove Lemma~\ref{L:follow-strip}.

\begin{proof}[of Lemma~\ref{L:follow-strip}]
We use Lemma~\ref{L:hyperbolic-metric} to endow $\Sigma$ with a hyperbolic metric.
By Lemma~\ref{L:hyperbolic-metric}(1), after applying $O(n^2)$ monotonic homotopy moves the resulting multicurve $\gamma'$ has length $O(n \log g)$.
Consider the set of disks centered at each point on the multicurve with radius $1/2$, which is smaller than half the systole by Lemma~\ref{L:hyperbolic-metric}(2); therefore all such disks are embedded in $\Sigma$.
Straighten any disk using Lemma~\ref{L:steinitz-ringel} if the tortuosity of the center point is at least $\e^2$.
Once every point on $\gamma'$ has tortuosity less than $\e^2$, by Lemma~\ref{L:tortuosity} the multicurve $\gamma'$ now lies in the $\e$-neighborhood of $\gamma_*$.

Straightening a disk takes $O(n^2)$ moves using Lemma~\ref{L:steinitz-ringel}.
The tortuosity at a center of each disk is a lower bound on the difference between the lengths of the multicurve $\gamma'$ before and after straightening.
From Lemma~\ref{L:hyperbolic-metric}(1) the length of $\gamma'$ is at most $O(n \log g)$.
Every time a disk is straightened the length of $\gamma'$ will drop by at least $\e^2$.
Since $\gamma'$ is non-contractible, the length of any curve homotopic to $\gamma'$ is at least the systole, which is $\Omega(1)$ by Lemma~\ref{L:hyperbolic-metric}(2).
Therefore at most $O(n \log g/\e^2)$ disks will be straightened before every point has tortuosity less than $\e^2$.
In total at most $O(n^3 \log g/\e^2)$ homotopy moves are performed.
From Lemma~\ref{L:strip-width}, setting $\e \coloneqq \Theta(g/(n \log g))$ concludes the proof of Lemma~\ref{L:follow-strip}.
\end{proof}

\section{Tightening Curves on Surface with Boundary}
\label{S:tightening-boundary}

In this section we prove Lemma~\ref{L:tighten-strip}.
Throughout the rest of the section, let $\Sigma$ be an orientable surface \emph{with} boundary and let $\gamma$ be a multicurve on $\Sigma$.

The second phase of the curve shortening algorithm by de Graaf and Schrijver~\cite{gs-mcmcr-97} starts with a multicurve $\gamma$ lying within an $\e$-neighborhood of its multigeodesic on $\Sigma$, where in some cases $\e$ is required to be exponentially small.
Unfortunately we cannot afford to drag $\gamma$ exponentially close to its geodesic which requires more than polynomially many moves (Section~\ref{S:move-to-geodesic}).
Instead, we make the observation that one can mimic this part of the algorithm in a combinatorial way, which, in particular, does not require the multicurve $\gamma$ to be close to its \emph{own} geodesic.
We pick an open neighborhood---called a \emph{pipe system}---of some underlying skeleton graph, such that $\gamma$ can be drawn in proper ways respecting the pipe system.
We then describe a way to morph the pipe systems using \emph{cluster} and \emph{pipe expansions},
a technique introduced by Cortese~\etal~\cite{cbpp-ecpg-09} in graph drawings (and later on applied to weak embeddings~\cite{weak,aaet-rwsp-17,aft-rweg-18} and crossing numbers~\cite{ft-cmpd-18}),
so that the multicurve inside the pipe system can be canonicalized using polynomially many monotonic homotopy moves.
Conceptually the expansion operations can be viewed as ways to morph the metric on surface $\Sigma$, so that curves on $\Sigma$ get transformed closer and closer to the geodesic with respect to the morphing metric.
After $\gamma$ is canonical we use the crossing minimization algorithm for flat braids to tighten $\gamma$~\cite{gp-icha-93,gs-mcmcr-97}.

We first define an initial pipe system using system of arcs, and the multicurve is then made to respect the pipe system in Section~\ref{SS:respect-pipes}.
In Section~\ref{SS:expansion} we introduce the expansion operations formally, followed by a description of the main algorithm and its analysis in Section~\ref{SS:expansion-algorithm}.

\subsection{Putting Curves into a Pipe System}
\label{SS:respect-pipes}

Let $G$ be a (multi-)graph drawn on a surface $\Sigma$ with boundary; we refer to the vertices and edges of $G$ as \EMPH{clusters} and \EMPH{pipes}.
The drawing of $G$ is not necessarily an embedding; assume without loss of generality that all self-intersections of $G$ are between its edges, are transverse and involve at most two edges.
A \EMPH{pipe system $\Pi$}%
\footnote{also known as strip system~\cite{aft-rweg-18,weak} or thickening~\cite{ft-cmpd-18}}
of $G$ is a topological neighborhood of the drawing of $G$ on surface $\Sigma$ with a decomposition into regions corresponding to clusters and pipes.
\begin{itemize}\itemsep=0pt
\item For each cluster $u$ in $G$, a \EMPH{cluster region $D_u$} is a topological disk containing $u$.
\item For each pipe $uv$ in $G$, a \EMPH{pipe region $R_{uv}$} is a topological disk containing $uv$ that is disjoint from the interior of the cluster regions $D_u$ and $D_v$.
Notice that if two pipes intersect in the drawing of $G$, then the two corresponding pipe regions cross on the surface $\Sigma$. However, three pipe regions are never allowed to overlap at any common point.
%
\item For each cluster $u$, there are disjoint connected subsets of the boundary of $D_u$ forming \EMPH{ends $A_{u,v}$}, one for each incident edge $uv$, in the order of the rotation system defined by the drawing of $G$; identify the intersection between $D_u$ and $R_{uv}$ with $A_{u,v}$.
\end{itemize}
When there is no risk of confusion, we sometimes refer to cluster and pipe regions as clusters and pipes as well.
Let $\Pi$ be the collection of all the cluster and pipe regions; from time to time we also abuse the notation and refer to the \emph{union} of all cluster and pipe regions as~$\Pi$, so that one can safely use sentences like ``(part of) $\gamma$ lies in the pipe system $\Pi$''.
If pipe system $\Pi$ is constructed from graph $G$, we refer to $G$ along with its drawing as the \EMPH{skeleton} of $\Pi$.
Each region can be viewed as a tangle; a \EMPH{strand} of a cluster or pipe is a maximal subpath of $\gamma$ inside the corresponding region.

It is easier to talk about the pipe system by imposing geometry to the topological disks: throughout this section, we will generally endow the cluster regions $D_u$ with the metric of a Euclidean disk, and the pipe regions $R_{uv}$ with the metric of a thin Euclidean rectangle.
However we emphasize that while some constructions and proofs in the following sections are described using geometry, they can be rephrased using purely combinatorial languages.

Now the plan is to construct an initial pipe system $\Pi_0$ using the system of arcs from Lemma~\ref{L:system-of-arcs}.
Let $\Xi$ denote the system of arcs given by Lemma~\ref{L:system-of-arcs}.
The \EMPH{pipe system $\Pi_0$} is obtained by taking the dual graph of $\Xi$ as a skeleton graph $G$, which consists of one unique cluster and $O(g+b)$ (self-loop) pipes.
But since we care about its precise position with respect to $\gamma$, we need to describe the construction of $\Pi_0$ more carefully.
We replace each arc of $\Xi$ by two identical copies infinitesimally close to each other, co-bounding a $4$-gon with two infinitesimal subpaths of the boundary.
Each of these $4$-gons is a pipe of the pipe system $\Pi_0$ while the ``big'' component corresponding to the unique polygon obtained by cutting $\Sigma$ along $\Xi$ is the single cluster of $\Pi_0$.
Note that $\gamma$ is trivially contained in the union of the regions of this pipe system, since this union is the whole surface $\Sigma$.

We will prove in Lemma~\ref{L:respect-pipes} that the multicurve~$\gamma$ can be made to \emph{respect} $\Pi_0$ by satisfying some good properties.
Such modified $\gamma$ along with the pipe system $\Pi_0$ will be the starting point of the algorithm.

\paragraph{Respecting pipe system.}
We say that a multicurve $\gamma$ \EMPH{respects} a given pipe system $\Pi$ if
\begin{enumerate}[(1)]\itemsep=0pt
\item $\gamma$ lies completely in $\Pi$;
\item all strands in any cluster or pipe region are simple and no two strands intersects more than once;
\item each component of the intersection between $\gamma$ and any end of $\Pi$ is a single transverse crossing;
\item $\gamma$ never intersects the same end consecutively more than once;
  in other words, whenever a curve enters a topological disk (whether it's a cluster or pipe region) from one end, the curve must leave from another end of the disk; and
\item within the intersection of a pair of pipe regions, pairs of strands from the same pipe region are not allowed to cross.
\end{enumerate}

Before we continue, we quickly comment that
given any multicurve $\gamma$ respecting a pipe system, one can safely assume the following additional property as part of the definition.
%
\begin{enumerate}[(6)]
\item No intersections of $\gamma$ are between strands of the same pipe.
\end{enumerate}

\begin{lemma}
\label{L:sliding}
Let $\Pi$ be a pipe system.
Let $R$ be a pipe region of $\Pi$, which is crossed transversely by other pipe regions $R_1, \ldots, R_k$. 
Let $\gamma$ be a multicurve with $n$ crossings respecting $\Pi$, but the strands of $\gamma$ may cross in the pipes.
Then one can find a sequence of $O(|R|^2)$
monotonic homotopy moves to push all the crossings between strands of $R$ to an incident cluster region,
where $|R|$ denotes the number of crossings between strands of $R$.
\end{lemma}


\begin{proof}
 We will push the crossings between strands of $R$ into an incident cluster region using a controlled number of $\arc33$ moves.
 A crossing between two strands of $R$ is called \textit{extremal} if, out of the four substrands that it defines, two of those do not cross any other strand of $R$ and end at the same end of $R$.
 Pick such an extremal crossing $z$ between two strands inside $R$: by orienting $R$ and its strands from one end to the other, leftmost and rightmost crossings will be extremal because of Properties~(2) and~(4).
 Now, the obstruction to simply moving~$z$ to an incident cluster comes from the other pipe regions crossing $R$ transversely, as their strands stand in the way.
 As part of the definition of a pipe system, no two pipe regions $R_i$ and $R_j$ intersects $R$ at a common point.

 By Property~(5), the strands of these transverse pipe regions do not cross within the intersection.
%
%
%
 Thus, the crossing $z$ can be pushed past strands of transverse pipes using only $\arc33$ moves.
 See Figure~\ref{F:sliding}.
 Since there are $O(|R|)$ crossings to push and each is pushed past $O(n)$ strands from transverse pipes (since each transverse strand induces at least once crossing), this can be done using $O(n^2)$ $\arc33$ moves.
\end{proof}

\begin{figure*}[h!]
  \centering
  \includegraphics[width=0.8\textwidth]{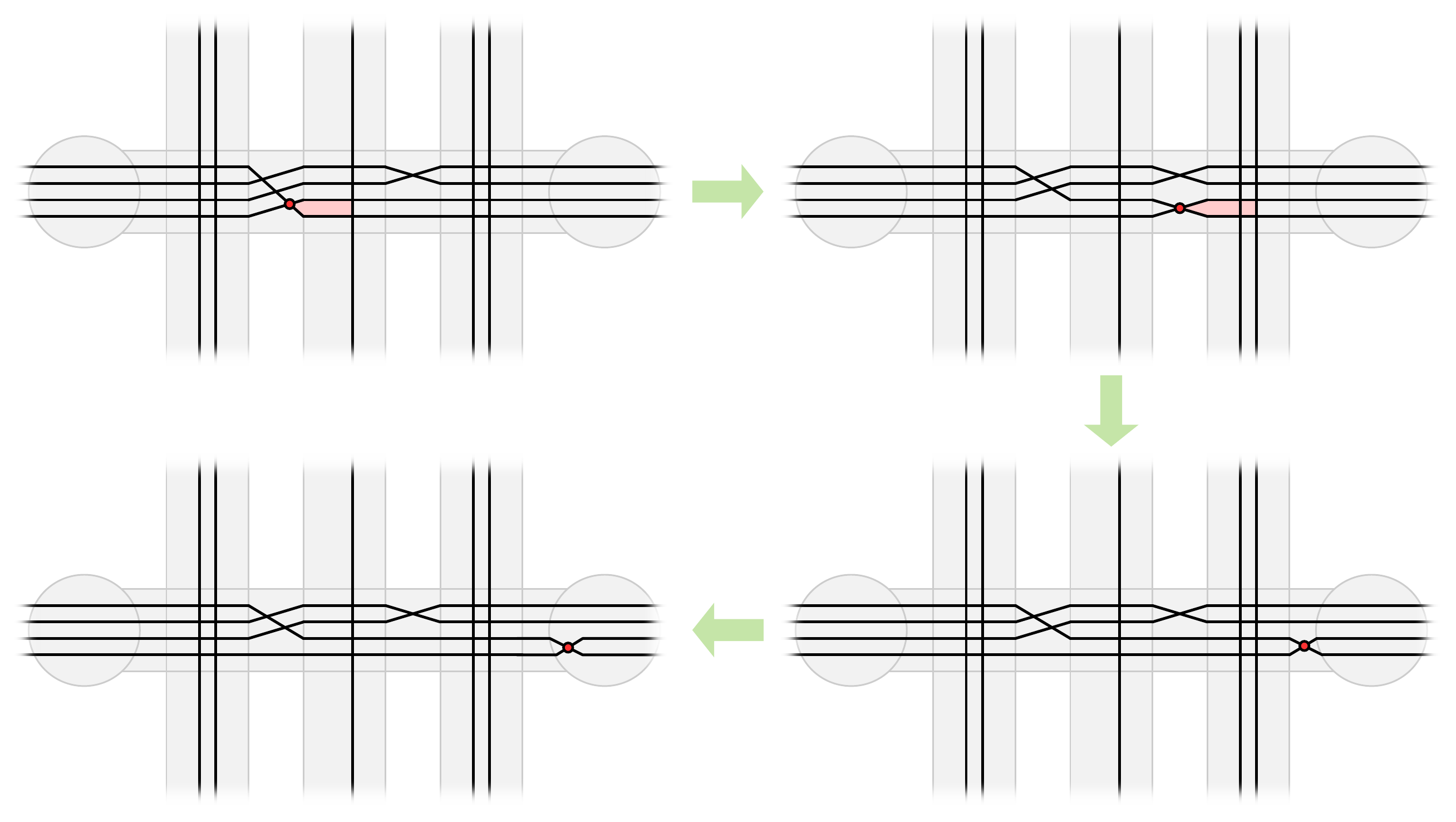}
  \caption{The red crossing within a pipe is pushed towards an incident cluster by doing $\arc33$ moves on the triangles marked in pink.
	}
  \label{F:sliding}
\end{figure*}

Using this lemma on all the pipe regions costs $O(n^2)$ moves.
Observe that all crossings of a multicurve $\gamma$ respecting a pipe system must either be inside the clusters, or between two intersecting pipes where the strands of the two pipes intersect in a grid-like pattern.
The strands inside a pipe can be drawn in parallel connecting from one end to the other, preserving their order on each end.

\paragraph{Closed walks on skeleton graph.}
If a multicurve $\gamma$ respects a pipe system $\Pi$ whose skeleton graph is $G$, one can define \EMPH{closed walks $C$} associated with $\gamma$ on $G$ as follows.
Let $\gamma_i$ be one of the constituent curve of $\gamma$.
Let
\[
A_{u_0,u_1}, A_{u_1,u_0}, A_{u_1,u_2}, A_{u_2,u_1}, \dots, A_{u_{w-1},u_0}, A_{u_0,u_{w-1}}
\]
be the sequence of ends of $\Pi$ that $\gamma_i$ intersects.  Then closed walk $C_{\Pi}(\gamma_i)$ in $G$ is defined to be
\[
[u_0, u_1, \dots, u_{w-1}, u_0].
\]
Closed walks are considered without basepoint, that is, up to cyclic permutations.
Closed walks $C_\Pi(\gamma)$ are defined to be the collection of all closed walks $C_{\Pi}(\gamma)$, each corresponding to a constituent curve $\gamma_i$ of $\gamma$.
Let \EMPH{$C(u)$} and \EMPH{$C(uv)$} denote the vertices and edges of closed walks $C$ that correspond to the cluster $u$ and pipe $uv$, respectively.
Let the \EMPH{weight $n(uv)$} be the number of times $C$ uses pipe $uv$ in $G$ (in either direction); one has $n(uv) = \abs{C(uv)}$.
Observe that the closed walks $C$ do not contain any \emph{spurs}---subwalks of the form $[u,v,u]$---by Property~(4) in the definition of respecting a pipe system.

Now we show that any multicurve $\gamma$ lying in the initial pipe system $\Pi_0$ can be made to respect it using a polynomial number of monotonic homotopy moves.

\begin{lemma}
\label{L:respect-pipes}
Let $\Sigma$ be a genus-$g$ orientable surface with $b$ boundary components.
Any multicurve $\gamma$ on $\Sigma$ with $n$ crossings can be made to respect the pipe system $\Pi_0$ using $O((g+b) n^2)$ monotonic homotopy moves.
Furthermore, the tightening problem remains unchanged: any tightening of $\gamma'$ within the pipe system is also a tightening of $\gamma$.
The length of the closed walk $C$ corresponding to $\gamma$ on $G$ is at most $O((g+b) n)$.
\end{lemma}

\begin{proof}
%
The very first step is to modify $\gamma$ so that it contains no embedded monogons or embedded bigons within $\Pi_0$.
This step follows from Steinitz algorithm (Lemma~\ref{L:steinitz}) and takes $O(n^2)$ moves, since removing each bigon or monogon takes $O(n)$ moves and this may need to be done $O(n)$ times.

We then use $O(|\gamma \cap \Pi_0| \cdot n)$ monotonic homotopy moves to ensure that $\gamma$ does not form any bigon with any end of $\Pi_0$.
This is identical to the first step in Chang \etal~\cite[Lemma~4.4]{untangle}; we repeat the main idea for clarity.
Assuming that there exists such a bigon, let $B$ denote a minimal embedded bigon (under containment) between $\gamma$ and an end of $\Pi_0$.
By minimality, the only subpaths of $\gamma$ occurring inside $B$ have to be simple and crosses $B$ transversely, from one side to the other.
Thus, we can remove $B$ by moving the subpath of $\gamma$ bounding $B$ across, going over each vertex one by one with a $\arc{3}{3}$ move.
We refer to Lemma~4.4 of Chang \etal~\cite{untangle} for more details and an illustration of this process.
Removing all the bigons between $\gamma$ and the ends of $\Pi_0$ using this technique costs $O(|\gamma \cap \Pi_0| \cdot n)$ monotonic homotopy moves.

We immediately have Property~(1) because $\gamma$ is contained in the union of the regions of $\Pi_0$.
Property~(2) follows from the fact that we first tightened~$\gamma$ to remove embedded monogons and bigons: if any strand of $\alpha$ was non-simple in a cluster or a pipe, it would form such a monogon or bigon.
Since $\gamma$ crosses the ends of $\Pi_0$ transversely, the crossings between $\gamma$ and any end of $\Pi_0$ is a point, yielding Property~(3).
We removed bigons between the ends $\Pi_0$ and $\gamma$, and thus $\gamma$ cannot intersect the same end of $\Pi_0$ consecutively more than once as it would yield a bigon. This gives Property~(4).
Property~(5) is true as the interiors of different pipe regions of $\Pi_0$ do not intersect.

As all the homotopy moves are performed within the pipe system $\Pi_0$, any tightening that can be obtained from the original $\gamma$ on the surface $\Sigma$ can also be realized by a tightening of the new $\gamma$ within $\Pi_0$ that covers the whole $\Sigma$.
Furthermore, based on the fact that each edge of $\gamma$ intersects $\Xi$ only $O(g+b)$ times, the length of the closed walks $C$ constructed from $\gamma$ will be at most $O((g+b)n)$.
Thus $\gamma$ can be made to respect $\Pi_0$ using $O(|\gamma \cap \Pi_0| \cdot n + n^2)=O((g+b)n^2)$ monotonic homotopy moves.
\end{proof}

\subsection{Tightening curves using local operations}
\label{SS:expansion}

We define two operations called the \EMPH{cluster expansion} and \EMPH{pipe expansion} performed on a multicurve~$\gamma$ lying in a pipe system $\Pi$ in this subsection.
Such operations have been used to study clustered planarity in graph drawings~\cite{cbpp-ecpg-09,ft-aecpt-2019}, weakly simple polygons~\cite{weak,aaet-rwsp-17}, weak embeddings of graphs~\cite{aft-rweg-18}, and crossing numbers~\cite{ft-cmpd-18}.
Our definition most closely resembles the one in Fulek and T\'oth~\cite{ft-cmpd-18}; both allow the edges of skeleton graph $G$ to cross in the drawing.
The main differences are, instead of preserving crossing numbers, we need to argue that the tightening problem remains the same before and after the expansion; and unlike the previous papers where expansions can be done instantly, we have to implement each expansion operation using monotonic homotopy moves.
While the constructions are described geometrically, the exact shape and position of the regions are mostly artificial and irrelevant; the only important thing is the change to the combinatorial structure.

\begin{figure*}[t!]
  \centering
  \includegraphics[width=0.8\textwidth]{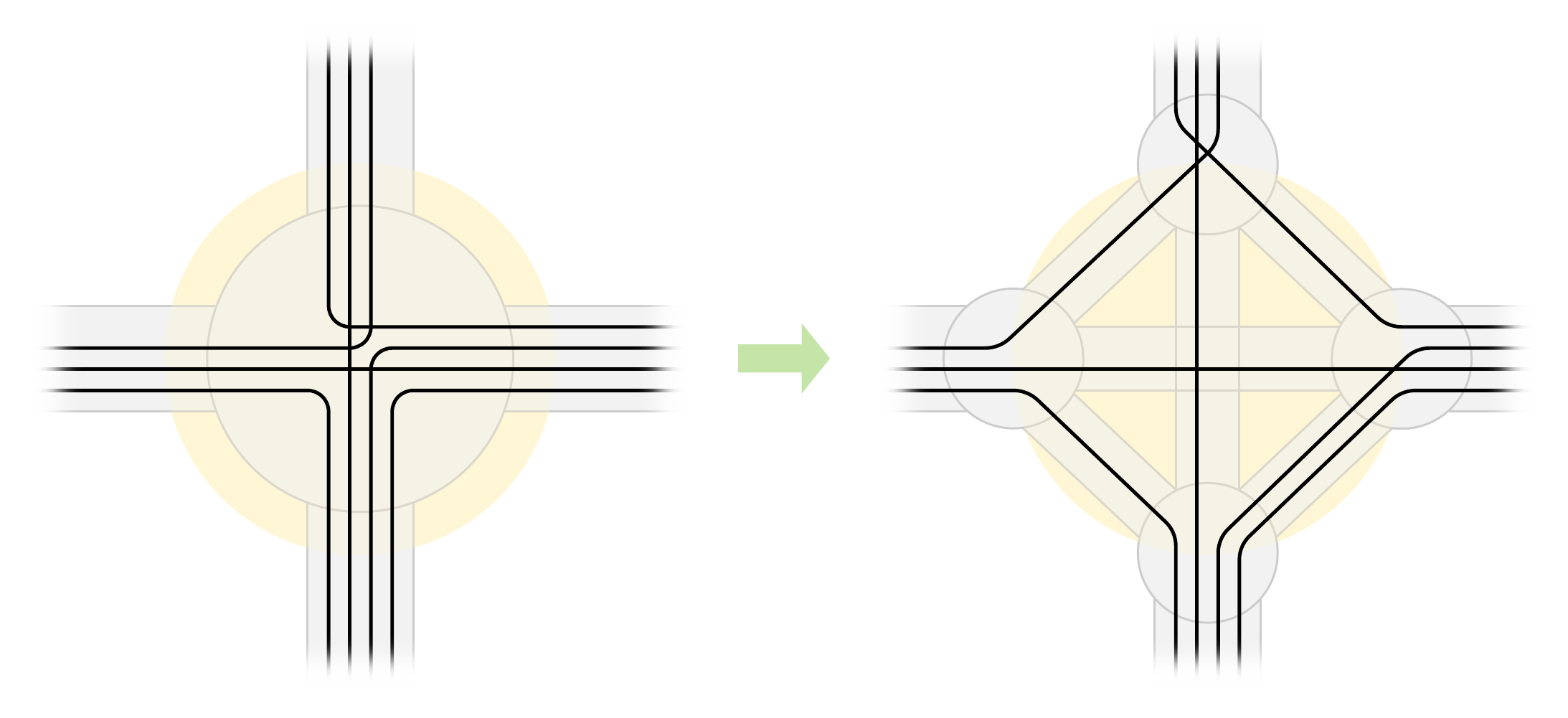}
  \caption{Expanding a cluster $D_u$. The infinitesimal ends have been widened for visibility.}
  \label{F:cluster-expansion}
\end{figure*}

\paragraph{Cluster expansion.}
We perform the following \EMPH{cluster expansion} on cluster $u$ and its region $D_u$ in a pipe system $\Pi$ with skeleton graph $G$.
To describe the construction, we endow $D_u$ with Euclidean metric such that all the ends are infinitesimally small, and position these ends so that there are no triple intersecting strands of $D_u$.

We modify $\gamma$ by replacing every strand of $D_u$ with a straight line, and modify $G$ and $\Pi$ accordingly (see Figure~\ref{F:cluster-expansion}):
For each pipe $uv$ incident to $u$, create a new cluster $[uv]$, whose corresponding cluster region in $\Pi$ is an elliptical neighborhood of the end $A_{u,v}$, and rename the pipe $uv$ to $[uv]v$.
For every pair of pipes $uv$ and $uw$ in $G$ insert a pipe $[uv][uw]$ if there was a strand of $\gamma$ that connects the end $A_{u,v}$ with the end $A_{u,w}$; insert the corresponding rectangular pipe regions in $\Pi$ accordingly, so that in the drawing of $G$ and $\Pi$ the pipes intersect transversely.
By the choice of Euclidean metric on $D_u$, these pipe regions can be taken to be arbitrarily thin and so that no three of them overlap at a point.
Finally, remove cluster $u$ from $G$ and region $D_u$ from $\Pi$.
Denote the multicurve and pipe system after cluster expansion as $\tilde\gamma$ and $\tilde\Pi$, respectively.

\begin{lemma}
\label{L:cluster-expansion}
First, the cluster expansion can be implemented using $O(n^2)$ monotonic homotopy moves, such that after cluster expansion the new multicurve $\tilde\gamma$ still respects the modified pipe system $\tilde\Pi$.
Second, if we denote by $\tilde\gamma_*$ any tightening of $\tilde\gamma$ within the modified pipe system $\tilde\Pi$ (viewed as a topological space), then $\tilde\gamma_*$ is also a tightening of $\gamma$ within $\Pi$.
\end{lemma}

\begin{proof}
Conceptually, this can be implemented by straightening the tangle defined by the cluster region using Lemma~\ref{L:steinitz-ringel} in $O(n^2)$ moves, as the number of crossings in $\gamma$ is upper bounded by $n$ at any point of the algorithm because the homotopy process is monotone.

Next we prove that the modified multicurve $\tilde\gamma$ respects the modified pipe system $\tilde\Pi$, by showing Properties~(1)--(5) (and thus also (6)).
Let the cluster expansion be performed on cluster $u$.
Properties~(1)--(3) immediately follow from the new strands being straight lines and the disk being convex.
For Property~(4), consider ends of two different types: ends of the form $A_{[uv],v}$ and ends of the form $A_{[uv],[uw]}$, where $v$ and $w$ are clusters adjacent to $u$ in $G$.
For ends of the first type, if $\gamma$ intersects $A_{[uv],v}$ consecutively twice, the subpath of $\gamma$ between the two intersections must lie inside $D_u$, and therefore must be a strand in $D_u$.  This implies that $\gamma$ did not respect $\Pi$ as Property~(4) was already violated, a contradiction.
For ends of the second type, as we took the new clusters to be elliptical neighborhoods of the ends, and the new strands are all straight lines between the ends, each such strand can cross the boundary of each cluster of $\tilde\Pi$ at most once.
Property~(5) follows from the fact that the pipe regions can be made infinitesimally thin, and thus generically there are no crossings within the intersection of two of those.

The second item of the lemma is a consequence of the facts that the cluster expansion performed on two multicurves with identical closed walks in the skeleton graph before expansion creates two new multicurves with identical closed walks in the new skeleton graph after expansion, and that homotopy between two multicurves within the pipe system is equivalent to the \emph{equality} between two corresponding closed walks in the skeleton graph.

More in details.
Since, as a topological space, $\tilde\Pi$ is obtained from $\Pi$ by adding punctures, any tightening $\tilde\gamma_*$ of $\tilde\gamma$ within $\tilde\Pi$ is also homotopic to $\tilde\gamma$ (and thus to $\gamma$) within $\Pi$.
So it suffices to prove that such a $\tilde\gamma_*$ is tight within $\Pi$.
In order to do so, let $\gamma_*$ denote a tightening of $\gamma$ within $\Pi$; we prove that $\gamma_*$ and $\tilde\gamma_*$ have equally many self-crossings.

First we claim that $\gamma_*$ can be made to respect the pipe system $\Pi$ using monotonic homotopy moves (no matter how many).
The proof follows closely the one of Lemma~\ref{L:respect-pipes}.
As $\gamma_*$ is tight, it contains no embedded monogons or bigons.
We ensure that it does not form any bigon with any end of $\Pi$ by undoing such bigons, starting from the innermost ones. As in that proof, $\gamma_*$ now satisfies Properties (1)--(4), but it may fail to satisfy Property~(5).
Within the intersection of two pipe regions, if two strands of $\gamma_*$ from the same pipe region cross, they define a trigon with a boundary of the region-intersection.
By applying Corollary~\ref{C:empty-trigon}, we can make this trigon empty using monotonic homotopy moves, and then move the crossing outside of the intersection.
This operation does not break Properties (1)--(4). Repeating this as many times as needed, $\gamma_*$ satisfies Property~(5) and thus respects the pipe system~$\Pi$.


Then we claim that the closed walk $C_\Pi(\gamma_*)$ is identical to $C_\Pi(\gamma)$.
Indeed, by Property~(4), none of these closed walks contains spurs, and $\gamma_*$ and $\gamma$ are homotopic by definition.
Since $\Pi$ retracts (as a topological space) into its skeleton graph $G$, the claim follows from the fact that two closed walks without spurs in a graph are homotopic if and only if they are identical (up to cyclic permutation).

We straighten the tangle induced by $\gamma_*$ within the cluster region using monotonic moves.
Now, $\gamma_*$ can be considered as a multicurve in the new pipe system $\tilde\Pi$, and furthermore, the closed walk $C_{\tilde\Pi}(\gamma_*)$ induced by $\gamma_*$ in $\tilde\Pi$ is identical to $C_{\tilde\Pi}(\tilde\gamma)$.
(Indeed, in general, for any curve $\alpha \subseteq \tilde\Pi \subseteq \Pi$, the closed walk $C_{\tilde\Pi}(\alpha)$ is simply obtained from $C_{\Pi}(\alpha)$ by replacing subwords $wuv$ with subwords $w[wu][uv]v$.
Because $C_\Pi(\gamma_*)=C_\Pi(\gamma)$ as shown above, one has $C_{\tilde\Pi}(\gamma_*)=C_{\tilde\Pi}(\tilde\gamma)$.)
From that we conclude that $\gamma_*$ and $\tilde\gamma$ are homotopic within $\tilde\Pi$.

Any curve in $\tilde\Pi$ that is tight in $\Pi$ is also tight in $\tilde\Pi$.
It follows that $\gamma_*$ is tight within $\tilde\Pi$, and is therefore a tightening of $\tilde\gamma$ within $\tilde\Pi$. Thus it has exactly as many self-crossings as any tightening $\tilde\gamma_*$ of $\tilde\gamma$ within $\tilde\Pi$.
This concludes the proof.
\end{proof}

\begin{figure*}[t!]
  \centering
  \includegraphics[width=0.9\textwidth]{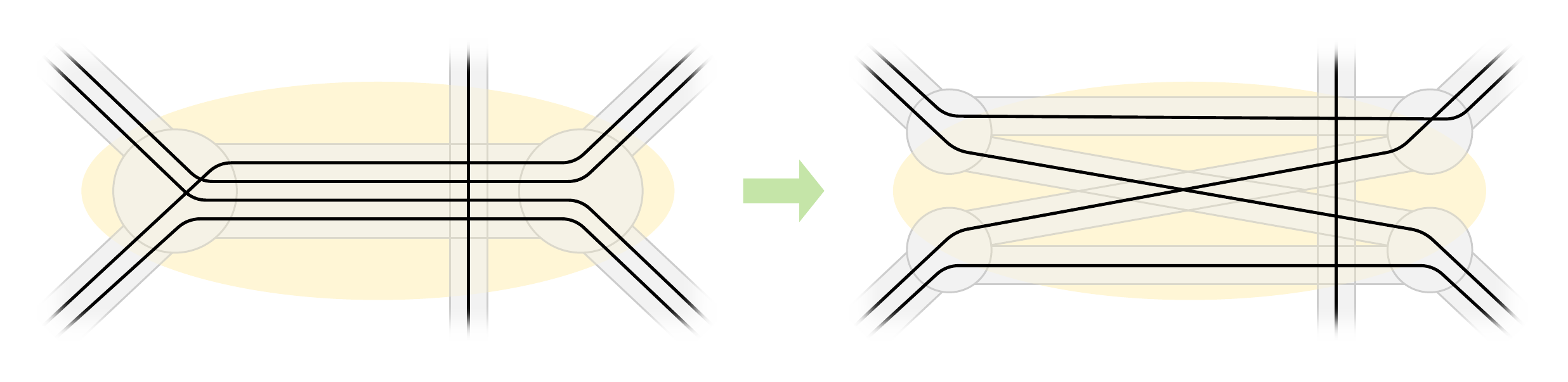}
  \caption{Expanding a pipe $R_{uv}$. The infinitesimal ends have been widened for visibility.}
  \label{F:pipe-expansion}
\end{figure*}

\paragraph{Pipe expansion.}
We perform the following \EMPH{pipe expansion} on a pipe $uv$.
%
For sake of analysis, we want to make sure that the pipe expansions we performed actually improve the quality of the multicurve in a pipe system.
This motivates the following definitions~\cite{cbpp-ecpg-09,weak-arxiv,ft-cmpd-18}.
Cluster $u$ is a \EMPH{base} of pipe $uv$ if every vertex in $C(u)$ is incident to some edge in $C(uv)$.
A pipe $uv$ is \EMPH{safe} if $u$ and $v$ are distinct and both $u$ and $v$ are bases of $uv$.
A pipe $uv$ in $G$ is \EMPH{useless} if both clusters incident to $uv$ have degree 2 in $G$;
otherwise the pipe $uv$ is \EMPH{useful}.
We will only perform pipe expansion on \emph{safe} and \emph{useful} pipes throughout the algorithm.

Let $\Delta$ be a topological ellipse containing the cluster regions $D_u$ and $D_v$ together with the pipe region $R_{uv}$, as well as portions of the pipe regions crossing $R_{uv}$ (see Figure~\ref{F:pipe-expansion}).
By taking $\Delta$ close enough to the region $R_{uv}$, we can assume $\Delta$ contains no intersections between any two pipes intersecting $R_{uv}$.
Because pipe $uv$ is safe, the two cluster regions are distinct.
We endow $\Delta$ with Euclidean metric such that all the ends are infinitesimally small, and position these ends so that there are no triple intersecting strands of $D_u$.

Then we modify $\gamma$ by replacing every strand of $\Delta$ with a straight line.
We modify $G$ and $\Pi$ accordingly, see Figure~\ref{F:pipe-expansion}:
For each pipe $uw$ incident to $u$ other than $uv$, create a new cluster $[uw]$, whose corresponding cluster region in $\Pi$ is a neighborhood of the intersection between pipe region $R_{uw}$ and the boundary of the disk $\Delta$.
Rename the pipe $uw$ into $[uw]w$.
Similarly for each pipe $vw$ incident to $v$, create a new cluster $[vw]$, whose corresponding cluster region in $\Pi$ is a neighborhood of the end $A_{v,w}$, and rename the pipe $vw$ into $[vw]w$.
Because the pipe $uv$ is safe, both $u$ and $v$ are bases of $uv$; and $G$ has no spurs as $\gamma$ respects $G$.
Therefore every strand of $R_{uv}$ must connect an end of $D_u$ to an end of $D_v$.
For every pair of pipes $uw$ and $vw'$, insert a pipe $[uw][vw']$ if there was a strand of $\Delta$ that connects end $A_{u,w}$ with end $A_{v,w'}$; insert a corresponding rectangular pipe region in $\Pi$ accordingly, so that in the drawing of $G$ and $\Pi$ the pipes intersect transversely.
By the choice of Euclidean metric on $\Delta$, these pipes can be taken to be arbitrarily thin rectangles such that no three of them overlap at a point.
%
Finally, remove clusters $u$ and $v$ from $G$ and regions $D_u$ and $D_v$ from $\Pi$.

\begin{lemma}
\label{L:pipe-expansion}
First, the pipe expansion can be implemented using $O(n^2)$ monotonic homotopy moves, such that after the pipe expansion the new multicurve $\tilde\gamma$ still respects the modified pipe system $\tilde\Pi$.
Second, if we denote by $\tilde\gamma_*$ a tightening of $\tilde\gamma$ within the modified pipe system $\tilde\Pi$ (viewed as a topological space), then $\tilde\gamma_*$ is also a tightening of $\gamma$ within $\Pi$.
\end{lemma}


\ifsoda
\else
The proof is virtually identical to the proof of Lemma~\ref{L:cluster-expansion}; we repeat it here to be comprehensive.

\begin{proof}
Let $uv$ be the useful pipe that we perform pipe expansion on.
Topologically, this can be implemented by straightening the tangle inside the disk $\Delta$ using Lemma~\ref{L:steinitz-ringel}, which can be done in $O(n^2)$ moves, as any crossing in the tangle is also a crossing of $\gamma$, which is upper bounded by $n$ at any point of the algorithm because the homotopy process is monotone.
Next we prove that the modified multicurve $\tilde\gamma$ respects the modified pipe system $\tilde\Pi$, by showing Properties (1)--(5) (and thus also (6)).
Properties (1)--(3) immediately follow from the new strands being straight lines and the disk being convex.
As for Property~(4), consider ends of two different types: ends of the form $A_{[uv],v}$ and ends of the form $A_{[uv],[vw']}$, where $w$ and $w'$ are clusters incident respectively to $u$ and $v$ in $G$.
For ends of the first type, if $\gamma$ intersects $A_{[uv],v}$ consecutively twice, the subpath of $\gamma$ between the two intersections must lie inside $D_u$, and therefore must be a strand in $D_u$.  This implies that $\gamma$ did not respect $\Pi$ as Property~(4) was violated, a contradiction.
For ends of the second type, as the new strands are all straight lines between the ends, and the cluster regions have been taken to be elliptical neighborhoods of the ends, each strand can cross the boundary of each cluster at most once.
Property~(5) follows from the fact that the pipe regions are infinitesimally thin, and thus generically there are no crossings within the intersection of two of those.

For the second item, let $\gamma_*$ denote a tightening of $\gamma$ within $\Pi$.
We claim that $\gamma_*$ can be made to respect the pipe system $\Pi$ using monotonic homotopy moves (no matter how many).
The proof follows closely the one of Lemma~\ref{L:respect-pipes}. As $\gamma_*$ is tight, it contains no embedded monogons or bigons. We ensure that it does not form any bigon with any end of $\Pi$ by undoing such bigons, starting from the innermost ones. As in that proof, $\gamma_*$ now satisfies Properties (1)--(4), but it may fail to satisfy Property~(5).
Within the intersection of two pipe regions, if two strands of $\gamma_*$ from the same pipe region cross, they define a trigon with a boundary of this intersection. By applying Corollary~\ref{C:empty-trigon}, we can make this trigon empty using monotonic homotopy moves, and then move the crossing outside of the intersection.
This operation does not break Properties (1)--(4). Repeating this as many times as needed, $\gamma_*$ satisfies Property~(5) and thus respects the pipe system $\Pi$.

We claim that the closed walk $C_\Pi(\gamma_*)$ that it induces there is identical to $C_\Pi(\gamma)$. Indeed, by Property~(4) of the pipe system, none of these closed walks contains spurs, and $\gamma_*$ and $\gamma$ are homotopic by definition.
Since $\Pi$ retracts (as a topological space) into its skeleton graph $G$, the claim follows from the fact that two closed walks without spurs in a graph are homotopic if and only if they are identical (up to cyclic permutation).

We straighten the tangles of $\gamma_*$ within the cluster region using monotonic moves.
Now, $\gamma_*$ can be considered as a curve of the new pipe system $\tilde\Pi$, and furthermore, the closed walk $C_{\tilde\Pi}(\gamma_*)$ induced by $\gamma_*$ in $\tilde\Pi$ is identical to $C_{\tilde\Pi}(\tilde\gamma)$.
(Indeed, in general, for any curve $\alpha \subseteq \tilde\Pi \subseteq \Pi$, the closed walk $C_{\tilde\Pi}(\alpha)$ is simply obtained from $C_{\Pi}(\alpha)$ by replacing subwords $wuvw'$ with the subwords $w[wu][vw']w'$.
Because $C_\Pi(\gamma_*)=C_\Pi(\gamma)$ as shown above, this implies $C_{\tilde\Pi}(\gamma_*)=C_{\tilde\Pi}(\tilde\gamma)$.)
From that we conclude that $\gamma_*$ and $\tilde\gamma$ are homotopic within $\tilde\Pi$.

Any curve in $\tilde\Pi$ that is tight in $\Pi$ is also tight in $\tilde\Pi$.
It follows that $\gamma_*$ is tight within $\tilde\Pi$, and is therefore a tightening of $\tilde\gamma$ within $\tilde\Pi$. Thus it has exactly as many self-crossings as any tightening $\tilde\gamma_*$ of $\tilde\gamma$ within $\tilde\Pi$.
This concludes the proof.
\end{proof}

\fi

\subsection{Main algorithm}
\label{SS:expansion-algorithm}

At the beginning of the algorithm, perform a cluster expansion on the unique cluster in the initial pipe system $\Pi_0$.
Note that after this first expansion, the two incident clusters of any pipe are distinct.
The algorithm repeatedly performs pipe expansion on an arbitrary safe and useful pipe,
until no such pipe remains.
Observe that after a cluster or pipe expansion, any newly created cluster is a base of an incident pipe.
Thus throughout the algorithm we maintain the invariant that every cluster is a base of some pipe; in particular, this implies that there is always a safe pipe after the first cluster expansion~\cite[Property~5]{cbpp-ecpg-09}~\cite[\S5.1]{weak-arxiv}~\cite[Lemma~5]{ft-cmpd-18}.

\ifsoda
\begin{lemma}[Chang \etal~{\cite[Lemma~5.3]{weak-arxiv}}]
\else
\begin{lemma}[Chang \etal~{\cite[Lemma~5.3]{weak-arxiv}}, Fulek and T\'oth~{\cite[Lemma~5]{ft-cmpd-18}}]
\fi
\label{L:no-useful}
If every cluster in the pipe system of $G$ is a base of an incident pipe, but $G$ has no useful pipes, then $G$ must be a disjoint union of cycles.
\end{lemma}

\begin{figure*}[t!]
\centering
    \includegraphics[width=0.6\textwidth]{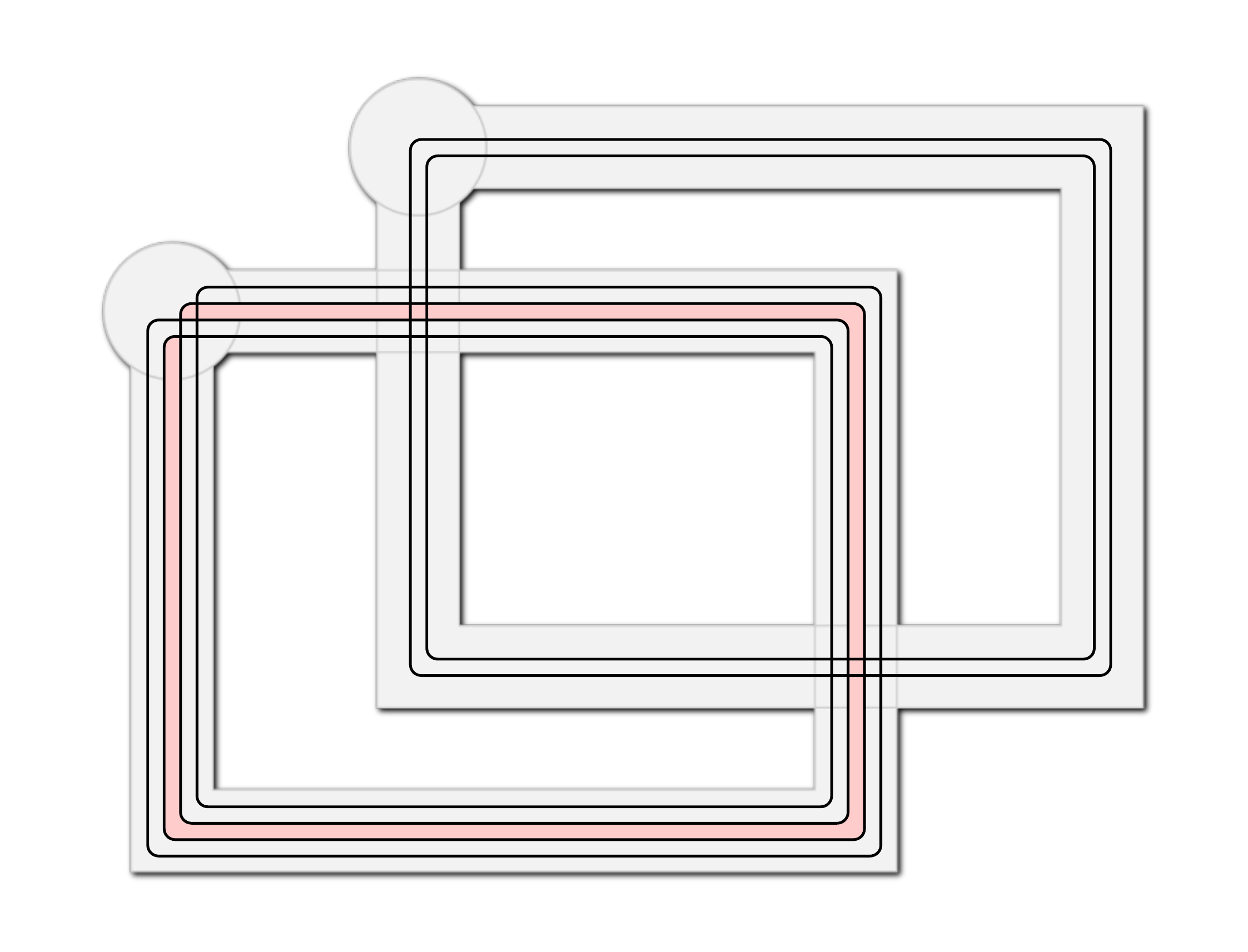}
    \caption{A multicurve $\gamma$ in a canonical form on a sphere with four boundaries.  Note that while the top-right curve is tightened, the bottom-left curve is not (an embedded bigon is labeled in pink).}
    \label{F:canonical}
\end{figure*}

By Lemma~\ref{L:no-useful}, when all the useful pipes are gone, the skeleton graph must be a disjoint union of cycles.
We now further put the multicurve into \emph{canonical form}.
A multicurve $\gamma$ respecting a pipe system $\Pi$ is said to be in \EMPH{canonical form} if
(1) the skeleton graph $G$ of $\Pi$ is a disjoint union of cycles;
(2) every intersection of $\gamma$ is either between two strands of different pipes, or lying in a unique cluster~$u^*$; furthermore, the strands in cluster $u^*$ forms (the projection of) a braid. See Figure~\ref{F:canonical}.

Let $\gamma$ be a multicurve respecting a pipe system whose skeleton graph is a disjoint union of cycles.
Slide all the crossings inside clusters to some arbitrarily chosen cluster $u^*$ using multiple applications of Lemma~\ref{L:sliding}, and perform Lemma~\ref{L:steinitz-ringel} on $u^*$ to straighten its strands~\cite[Propositions~14]{gs-mcmcr-97}; this can be done using $O(n^2)$ many monotonic homotopy moves.
Now the multicurve must be in canonical form.
De Graaf and Schrijver~\cite[Propositions~8 and 14]{gs-mcmcr-97} describe an algorithm to tighten any multicurve in canonical form using quadratically many monotonic homotopy moves.
The main technical lemma~\cite[Proposition~7]{gs-mcmcr-97} can be described as a way to reduce the number of transpositions required to represent any permutation using conjugations.
Intuitively, they show that as long as the multicurve is not tightened, for each constituent curve $\gamma_i$ there is always a crossing that one can slide from the top of the braid all the way along the pipes where $\gamma_i$ lies in, to the bottom of the braid and cancels out with another crossing using a $\arc20$ move, possibly after a sequence of $\arc33$ moves to adjust the position of the multicurve.

\paragraph{Analysis.}
Define the \EMPH{potential function $\Phi(C,G)$} to be the number of edges in $C$ minus the number of pipes in $G$; in notation, $\Phi(C,G) \coloneqq \abs{E(C)} - \abs{E(G)} = \sum_{uv} \Paren{\abs{C(uv)} - 1}$.

\ifsoda
\else
\begin{lemma}[Chang \etal~{\cite[Lemma~5.4]{weak-arxiv}}, Fulek and T\'oth~{\cite[Lemma~4]{ft-cmpd-18}}]
\label{L:potential}
Let $\gamma$ be a multicurve respecting a pipe system $\Pi$, with corresponding closed walks $C$ on skeleton graph $G$.
Each pipe expansion performed on a useful pipe in $\Pi$ decreases the potential $\Phi(C,G)$ by at least one.
\end{lemma}
\fi

We now have all the tools to prove Lemma~\ref{L:tighten-strip}.

\begin{proof}[of Lemma~\ref{L:tighten-strip}]
As the potential $\Phi$ is always nonnegative and the initial value is at most $O((g+b)n)$ by Lemma~\ref{L:respect-pipes}, the algorithm terminates after $O((g+b)n)$ steps by Lemma~\ref{L:potential}.
Each cluster and pipe expansion can be implemented by $O(n^2)$ many homotopy moves (Lemmas~\ref{L:cluster-expansion} and~\ref{L:pipe-expansion}).
After no useful pipe remains, by Lemma~\ref{L:no-useful} the skeleton graph must be a disjoint union of cycles.
Turning $\gamma$ into canonical form via Lemma~\ref{L:sliding}, followed by tightening $\gamma$, using the algorithm by de Graaf and Schrijver~\cite{gs-mcmcr-97}, takes $O(n^2)$ many homotopy moves.

By Lemmas~\ref{L:cluster-expansion} and~\ref{L:pipe-expansion}, a multicurve that is tight in the final pipe system is also tight in the original pipe system, which covers the entire surface.
Therefore we have tightened the multicurve $\gamma$, and this proves Lemma~\ref{L:tighten-strip}.
\end{proof}

\section{Applications}
\label{S:applications}

\subsection{Putting curves in minimal positions}
\label{SS:minimal}

In this subsection, we explain how our techniques can be leveraged to prove Theorem~\ref{Th:minimal}.
Most of the steps of the proof are readily algorithmic, in particular the techniques of Section~\ref{S:tightening-boundary} have been designed with graph drawing applications in mind.
While we have described many steps using geometry, this was just for convenience and the cluster and pipe expansion operations are purely combinatorial: instead of endowing the disks with a Euclidean metric, one can just choose an (arbitrary) arrangement of pseudolines describing how the infinitesimally thin pipe regions will cross.

In order to do the expansions, we rely on Lemma~\ref{L:steinitz-ringel}.
While one could analyze the complexity of computing the right sequence of moves, let us just observe that for our application in mind---to compute a minimal position for the given multicurve---we can directly straighten the tangle within the disk.
So the complexity of straightening a tangle with $n$ crossings and $m$ strands is simply $O(n+m)$. Each cluster and pipe expansion can be realized in $O((g+b)n)$ time since there are $O((g+b)n)$ strands and $O(n)$ crossings, therefore the time complexity of the steps described in Section~\ref{S:tightening-boundary} is $O((g+b)^2n^2)$.

\medskip
The remaining algorithmic question is to carry out the straightening process described in Section~\ref{S:move-to-geodesic}, which relies heavily on hyperbolic geometry.
Hyperbolic geometric computations are becoming increasingly common (see for example Iordanov and Teillaud~\cite{iordanovteillaud}) in computational geometry, and the ones we rely on are no harder to carry out than the Euclidean ones.
As is customary in computational geometry, we work in a real RAM model. For the sake of convenience, we allow computations of hyperbolic trigonometric functions and their inverses in constant time, but note that this is not strictly required as approximating those functions works equally well.

The first step of Section~\ref{S:move-to-geodesic} is to endow the surface with a hyperbolic metric.
In terms of algorithms, this is achieved by mapping the surface to a polygon in some model of hyperbolic geometry, for example the Poincaré disk.
The vertices of the cut graph are mapped to the vertices of a regular hyperbolic polygon, and can thus be computed readily.
The edges in the Poincaré disk model are portions of circles intersecting the boundary at right angles, and can therefore be encoded by their endpoints.
Since the second step of the proof of Section~\ref{S:move-to-geodesic} is to straighten the tangle of $\gamma$ within this polygon, for an algorithm we can start directly with a straightened tangle, that is, a collection of circular arcs.

The main loop of the algorithm consists of taking a point of maximal tortuosity and straightening the tangle within a disk centered at this point.
Since, throughout the tightening process, the curve $\gamma$ stays a piecewise-geodesic, the points of maximal tortuosity are always at the vertices between two consecutive circular arcs.
We loop on all these points to find the one with largest tortuosity.
This requires length computations which can be done in constant time with inverse hyperbolic functions (note that approximating computations are also fine here, and thus adding such functions to the computation model is not strictly required).

We straighten the tangle in a disk centered at the point with maximum tortuosity, and each time such a straightening is done, the number of breaking points increases by one for each strand in the tangle.
There is a potential pitfall here: each straightening could double the number of breaking points in the multicurve, which could lead to an exponential number of circular arcs.
Furthermore, the number of intersections between the multicurve and the reduced cut graph might also increase during a straightening.
However, we can prove that they do not increase too much.

\begin{lemma}
Let $\Sigma$ be endowed with the hyperbolic metric described in Section~\ref{S:move-to-geodesic} for a cut graph $X$, and let $\gamma$ be an $n$-vertex multicurve on $\Sigma$ that crosses $X$ at most $O(n)$ times such that all its segments between successive intersections with $X$ are geodesics.
After any number of straightenings of $\gamma$ within a disk or within the polygon $\Sigma \setminus X$, the number of intersections of $\gamma$ with $X$ is $O(gn)$.
\end{lemma}

\begin{proof}
When endowed with the hyperbolic metric described in Section~\ref{S:move-to-geodesic}, the edges of the cut graph are geodesic segments.
We extend these segments into a family of $O(g)$ closed geodesics, which we denote by~$\Delta$.
These extensions are closed because the hyperbolic metric has been obtained from an equilateral $O(g)$-gon, and thus extending the geodesics just amounts to triangulating the polygon by adding a vertex at the center and connecting it to every vertex by geodesics.
Since $\gamma$ is geodesic within the polygon defined by~$X$, it is made of $O(n)$ geodesic arcs, each of which crosses $\Delta$ at most $O(g)$ times, thus the number of intersections between $\gamma$ and $\Delta$ is $O(gn)$. We claim that this number of intersections does not increase after any sequence of straightenings.
Indeed, we can consider that each straightening is a straightening of $\gamma \cup \Delta$, viewed as a single multicurve.
Since $\Delta$ is already made of geodesics, it is unchanged by the straightening. Since the straightening does not increase the number of crossings of the multicurve, the number of intersections of $\gamma$ with $\Delta$, and thus with $X$, is $O(gn)$.
\end{proof}

Therefore, by inserting a straightening of all the strands within $\Sigma \setminus X$ between two straightenings within a disk, we can ensure that the multicurve $\gamma$ is always made of $O(gn)$ geodesic segments, which we can encode using the coordinates of their endpoints.

The loop ends when no point of high tortuosity is found.
It remains to find a point inside the hyperbolic polygon which is far away (under the hyperbolic metric) from a collection of polynomially many geodesic arcs.
As we emphasized after the proof of Lemma~\ref{L:tortuosity}, in hyperbolic geometry, the sets of points equidistant to a geodesic are called \textit{hypercircles}, which are \textit{not} geodesics (unlike in the Euclidean setting), but they are still realized by circular arcs in the Poincaré disk model.
Therefore, one can easily compute all these hypercircles, and find a point to puncture in the complement of regions bounded by pairs of circles.

The bottleneck of the algorithm to move curves close to the geodesics is the sequence of straightenings in disks centered at points of high tortuosity: there are $O(n^3 \log^3 g/g^2)$ such straightenings.
Each of them requires finding the correct breaking points among the $O(gn)$ possibilities, after which straightening the tangle also takes $O(gn)$ time.
The resulting algorithm runs in $O(n^4 \log^3 g/g)$ time.
In total, we can put a multicurve in minimal position in $O(n^4 \log^3 g/g+(g+b)^2n^2)$ time.

\paragraph{Handling the torus.}
While our techniques do not apply to multicurves on a torus without boundary, this case can be handled manually using homology (see for example Stillwell~\cite[Chapter~5]{s-ctcgt-93}).
Indeed, the homotopy class of a closed curve on a torus coincides with its homology, which is a pair of integers $(p,q) \in \mathbb{Z}^2$.
For a given multicurve $\gamma$, one can compute the corresponding collection of pairs of integers $(p_i,q_i)$ in polynomial time by using any classical homology computation algorithm.
Then a minimal position of the multicurve $\gamma$ can be obtained by drawing~$\gamma$ on a Euclidean flat torus represented by a unit-square, where each closed curve is realized by a straight geodesic of slope $q_i/p_i$.
If $\mathrm{gcd}(p_i, q_i) \neq 1$, we take $\mathrm{gcd}(p_i, q_i)$ many copies of the geodesic with a slight offset and connect them so as to realize the homology class $(p,q)$ with exactly $\mathrm{gcd}(p_i, q_i)-1$ self-crossings. 
This multicurve is homotopic to $\gamma$ because it is homological to $\gamma$, and on the torus homology and homotopy coincide.
It is in minimal position because any closed curve of homology $(p,q)$ has at least $\mathrm{gcd}(p, q)-1$ self-crossings, and the crossings between different components realize the \textit{algebraic intersection number} $p_i q_j - p_j q_i$, of which the absolute value is known to lower bound the geometric intersection number.


\subsection{Electrical reductions}
\label{SS:electrical}

\begin{figure*}[t!]
\centering
\includegraphics[width=0.75\textwidth]{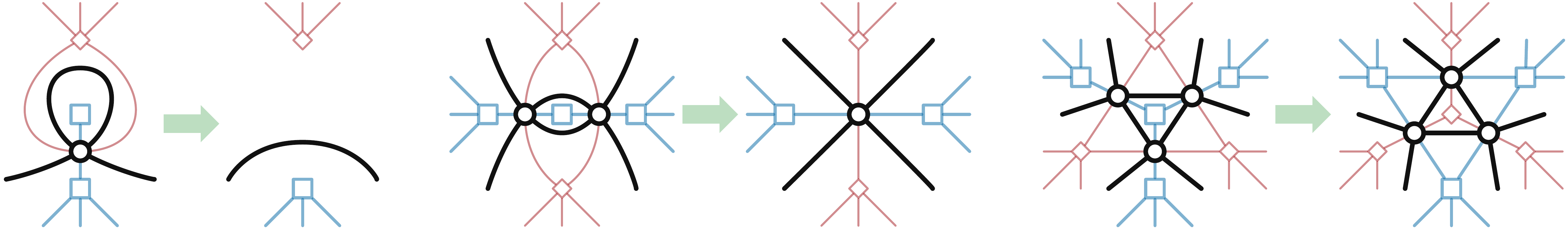}
\caption{Electrical moves $\arc10$, $\arc21$, and $\arc33$.}
\label{F:medial-elec}
\end{figure*}

Next, we consider the implication on reducing surface graphs using electrical transformations and sketch a proof of Theorem~\ref{Th:electrical}.
As mentioned in the introduction, every $k$-terminal graph embedded on a surface $\Sigma$ corresponds to a multicurve lying on a \emph{punctured} surface through the medial construction by adding one puncture on $\Sigma$ for each terminal in the graph.
The set of facial electrical transformations corresponds to a set of local operations on multicurves called the \EMPH{electrical moves}, which bear extreme similarity to the homotopy moves.
Let \EMPH{$X(\gamma)$} denote the minimum number of electrical moves required to tighten $\gamma$ on $\Sigma$, and let \EMPH{$H^\downarrow(\gamma)$} denote the minimum number of monotonic homotopy moves required to tighten $\gamma$ on $\Sigma$, without ever increase the number of vertices.
One can prove that if $H^\downarrow(\gamma) \le f(n)$ for any $n$-vertex multicurve on surface $\Sigma$, then $X(\gamma) \le n\cdot f(n)$ holds, by replacing the first $\arc20$ move with a $\arc21$ move and recurse~\cite[Lemma~7.2]{changthesis}.
In particular, our polynomial upper bound on the number of monotonic homotopy moves directly implies a polynomial bound on the number of electrical moves required to tighten a multicurve.
The original proof can be turned algorithmic by using the polynomial-time algorithm for monotonic homotopy from Theorem~\ref{Th:minimal}; we include a proof here to make the presentation complete.

\begin{lemma}[Chang~{\cite[Lemma~7.2]{changthesis}}]
\label{L:homotopy-electric}
Fix an arbitrary surface $\Sigma$.
Any polynomial running time $f(n)$ for tightening any $n$-vertex multicurve $\gamma$ on $\Sigma$ using monotonic homotopy moves can be turned into a polynomial running time $n \cdot f(n)$ for tightening any $n$-vertex multicurve $\gamma$ on $\Sigma$ using electrical moves.
\end{lemma}

\begin{proof}
Let $\gamma$ be an arbitrary multicurve on $\Sigma$ with $n$ vertices.
First, if multicurve $\gamma$ is already tightened under homotopy moves, then it must be tight under electrical moves as well (see Chang, Cossarini, and Erickson~\cite[Lemma~3.7]{homoelectric}) and thus the statement trivially holds.
Otherwise, consider the first homotopy move in the algorithm for tightening $\gamma$ monotonically (from Theorem~\ref{Th:minimal}) that decreases the number of vertices in $\gamma$ (that is, either a $\arc10$ or $\arc20$ move).
Replace the $\arc20$ move with a $\arc21$ if needed, one arrives at a curve $\gamma'$ that has strictly fewer vertices than $\gamma$.
We restart the monotonic homotopy algorithm on $\gamma'$ instead.

The time spent from reducing $\gamma$ to $\gamma'$ is at most $f(n)$.
Now by induction on the number of vertices, the time $g(n)$ for reducing $\gamma$ is
\ifsoda
\else
\begin{align*}
	g(n)
	&~\le~ g(n-1) + f(n) \\
	&~\le~ (n-1) \cdot f(n-1) + f(n) \\
	&~\le~ n \cdot f(n),
\end{align*}
which proves the lemma.
\fi
\end{proof}

\ifsoda
\else
\subsection{Terminal-leaf contraction}
\label{SS:terminal-leaf}

\begin{figure*}[t!]
\centering
\includegraphics[width=0.6\textwidth]{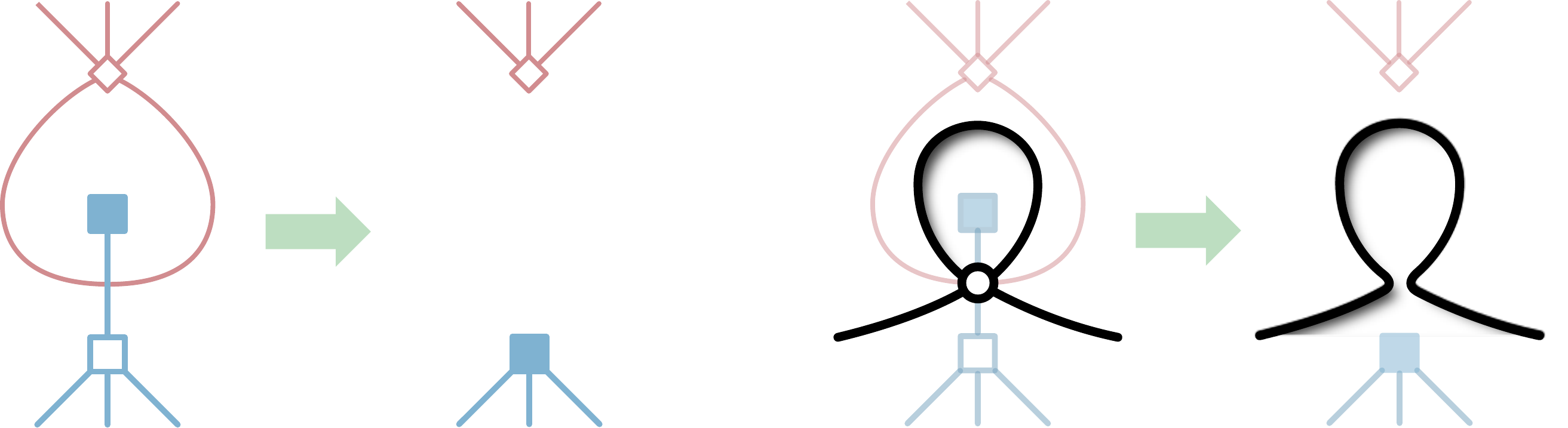}
\caption{A terminal-leaf contraction on graph $G$ and the corresponding operation on medial multicurve $G^\times$; the blue solid squares represent the terminals in $G$, which turn into punctures on the surface.}
\label{F:TL}
\end{figure*}
\fi


Previous works~\cite{fp-dtert-93,acgp-frpwg-00,gs-tdrpg-11,dm-ftpdw-15} on reducing surface graphs with terminals using electrical transformations also rely on an additional move which is not a facial electrical move.
The \EMPH{terminal-leaf contraction}%
\footnote{also known as the \emph{FP-assignment}~\cite{g-dtaa-91}, after Feo and Provan~\cite{fp-dtert-93}} is a leaf-reduction performed on a terminal by contracting its unique incident edge and assigning the merged vertex as a new terminal.
%
\ifsoda
\else
See Figure~\ref{F:TL} for a graph- and curve-view of this operation.
First thing to notice is that a multicurve that is tight under electrical moves might not be tight when terminal-leaf contractions are allowed.
We will argue that after the multicurve is tightened using electrical moves, we can further reduce the curve efficiently using terminal-leaf contractions, until nothing can be done.
This yields the desired polynomial bounds on number of monotonic homotopy moves, as well as a polynomial-time algorithm.
To this end we first introduce the following concepts.

\begin{figure*}[t]
  \centering
  \includegraphics[width=0.3\textwidth]{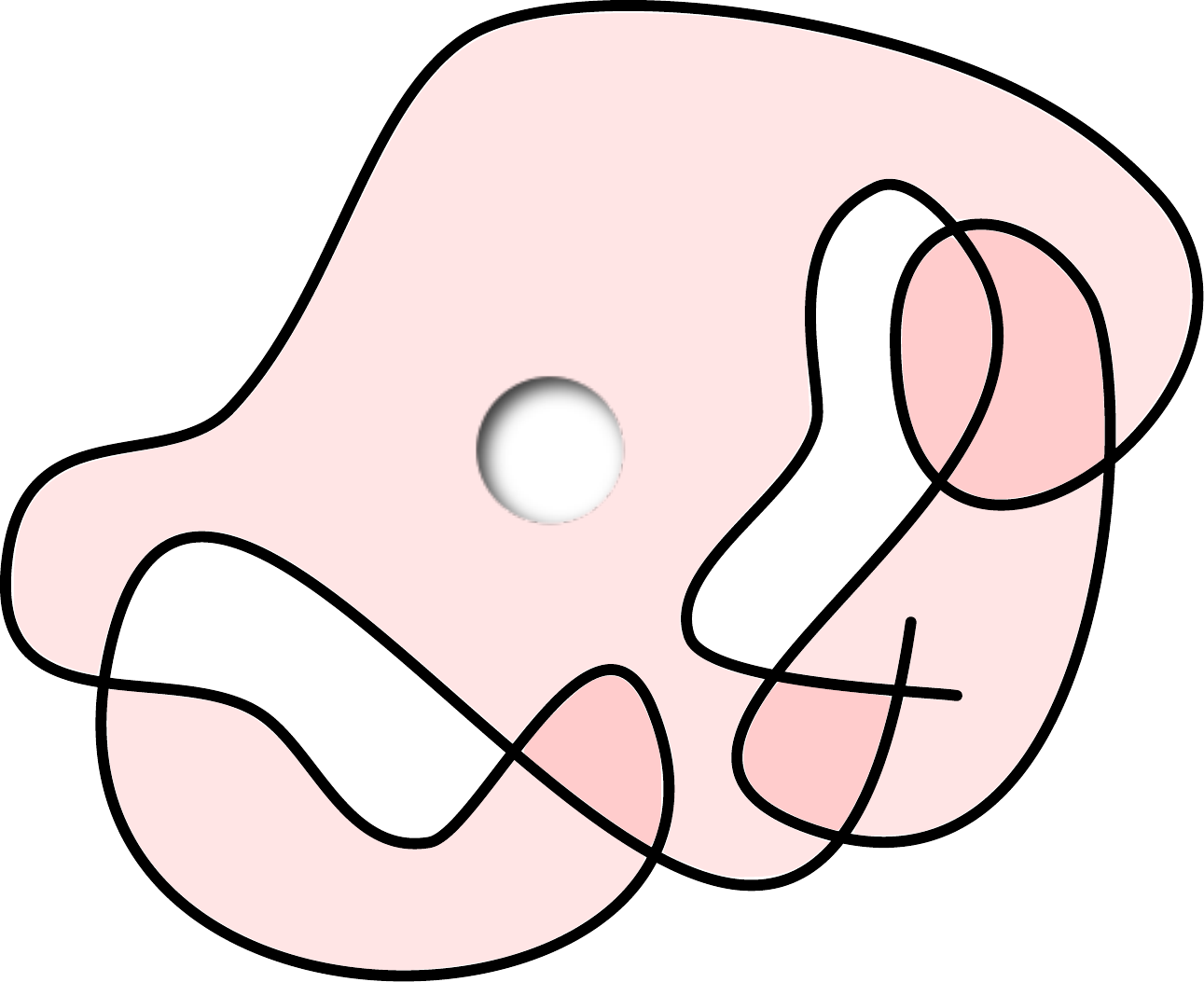}
  \caption{A punctured monogon $\alpha$ and the image of its singular disk.}
\label{F:punctured-monogon}
\end{figure*}

\paragraph{Punctured monogon.}
A \EMPH{smoothing} of a multicurve $\gamma$ at a vertex~$x$ replaces the intersection of $\gamma$ around $x$ with two disjoint simple paths, so that the result is another multicurve; notice that there are two possible smoothings at each vertex.
More generally, a \emph{smoothing} of a multicurve $\gamma$ is any multicurve obtained by smoothing a subset of its vertices.
A multicurve $\gamma$ contains a (singular) \EMPH{punctured monogon} $\alpha$ if $\gamma$ has a constituent curve $\gamma_i$ with a crossing $x$ such that after smoothing $\gamma_i$ at $x$ in the way that disconnects $\gamma_i$, one of the two curves of the smoothing is freely homotopic to a small circle going exactly once around a single puncture of $\Sigma$.
We call $x$ the \EMPH{tip} of punctured monogon $\alpha$.
This small circle bounds an embedded punctured disk, which turns into a singular punctured disk under the homotopy.
(This is akin to diagrams used in geometric and combinatorial group theory, see e.g.~\cite{ls-cgt-01}.)
The monogon~$\alpha$ itself designates this singular disk, that is, a map $\alpha\colon D^2 \rightarrow \Sigma$ such that $\alpha(D^2)$ covers the puncture exactly once.
See Figure~\ref{F:punctured-monogon} for an example.
A punctured monogon $\alpha$ is \EMPH{empty} if $\alpha(\bd D_2)$ is simple and the interior of $\alpha(D^2)$ is disjoint from $\gamma$.
A (medial) terminal-leaf contraction always applies on empty punctured monogons.
By definition punctured monogons are exactly those objects that are homotopic to some empty punctured monogon.

As a sanity check to see that punctured monogons are the right objects to work with, observe that any tight multicurve on a sphere with three punctures can be made simple by removing all the punctured monogons; this corresponds exactly to the statement that any 3-terminal plane graph can be reduced completely using electrical transformations along with terminal-leaf contractions~\cite{g-dtaa-91}.

\medskip
First, we show that any punctured monogon can be removed efficiently.

\begin{lemma}\label{L:sing-monogons}
  Let $\gamma$ be a multicurve with $n$ crossings in minimal position on a surface $\Sigma$. If $\gamma$ contains a punctured monogon, then we can compute in $O(n^3)$ time a sequence of $O(n^3)$ $\arc33$ moves transforming $\gamma$ into a multicurve $\gamma'$ containing an empty punctured monogon.
\end{lemma}

\begin{proof}
Let $\alpha$ be a punctured monogon in $\gamma$.
Our plan for the proof is to first turn $\alpha$ into an embedded disk (with a puncture), then empty it.
Smooth the vertex $x$ in $\gamma$ (by removing a small neighborhood of $x$ in $\gamma$ and reconnecting the curves in a non-crossing fashion) so that the punctured monogon $\alpha$ turns into a closed curve, which we denote as \EMPH{$\check\alpha$}.
By definition of punctured monogon, $\check\alpha$ is homotopic to a simple curve winding around the puncture.
A result of Hass and Scott~\cite[Theorem~2.4]{hs-ics-85} shows that if $\check\alpha$ is not simple, there must be an embedded bigon or monogon in $\check\alpha$ that contains $x$ on the boundary (because $\gamma$ itself is in minimal position).
Therefore such a bigon (resp.\ monogon) corresponds to a trigon (resp.\ bigon) in $\gamma$.
By minimality of $\gamma$, the only possibility is that this is a trigon.
Using Lemma~\ref{L:steinitz} and Corollary~\ref{C:empty-trigon} we can empty and flip the trigon in $\gamma$ (and thus the bigon in $\check\alpha$) using at most $O(n^2)$ moves.
Now the complexity of $\check\alpha$ decreases; after at most $n$ steps we turn $\check\alpha$ into a simple closed curve.
We emphasize that as $\gamma$ is already tight, and the complexity of $\gamma$ does not change throughout the process.

At this stage, the singular disk of $\alpha$ is embedded.
If we consider the surface $\Sigma$ without the puncture, $\alpha$ forms an embedded monogon.
Let $\alpha'$ be an innermost embedded monogon inside $\alpha$ (possibly $\alpha=\alpha'$).
Considering $\Sigma$ again with the puncture, since $\gamma$ is in minimal position, $\alpha'$ must be a punctured monogon.
Since it is innermost, all the strands crossing $\alpha'$ are simple.
Any such strand bounds a trigon with the tip of $\alpha'$, which we can remove once again using Corollary~\ref{C:empty-trigon}. This does not increase the number of crossing strands with the $\alpha'$.
Thus, removing all such strands costs $O(n^3)$ moves, after which the punctured monogon $\alpha'$ is empty.
\end{proof}

\def\Route{\operatorname{\mathit{route}}}

\paragraph{Routing set.}
Next we show that when there are no more punctured monogons, the multicurve must be tight under electrical moves and terminal-leaf contractions.
For any multicurve $\gamma$, the \EMPH{routing set}~\cite{homoelectric} of $\gamma$ is the following collection of homotopy classes:
\[
\EMPH{$\Route(\gamma)$} \coloneqq \Set{\Big. \Brack{\check\gamma} \mid \text{$\check\gamma$ is a smoothing of $\gamma$}}.
\]
Each homotopy class in $\Route(\gamma)$ is referred to as a \EMPH{route} of $\gamma$.
A key property is that the routing set is invariant under electrical moves~\cite[Lemma~3.6]{homoelectric}.

\begin{lemma}
\label{L:tightness}
Let $\gamma$ be a connected multicurve in minimal position on a surface $\Sigma$. If $\gamma$ does not contain a punctured monogon, then no sequence of electrical moves and terminal-leaf contractions can transform~$\gamma$ into another multicurve with fewer crossings than $\gamma$.
\end{lemma}

\begin{proof}
Assume otherwise that there is a sequence of electrical moves and terminal-leaf contractions transforming $\gamma$ into another multicurve with fewer crossings.  Let the number of terminal-leaf contractions used in this sequence be minimal among all counterexamples.
If no terminal-leaf contraction is used, then the fact that $\gamma$ is tight under electrical moves is a theorem of Chang, Cossarini and Erickson~\cite[Lemma~3.7]{homoelectric}.

Otherwise, denote by $\gamma'$ and $\gamma''$ the two multicurves just before and after the first terminal-leaf contraction is used.
The curve $\gamma'$ contains an empty punctured monogon, whose tip is denoted by $x$.
We emphasize that $\gamma'$ might not be in minimal position.
We first claim that \emph{$\gamma$ is homotopic to a smoothing of $\gamma''$.}
By definition, $\gamma'$ is related to $\gamma$ by electrical moves, thus $\Route(\gamma')=\Route(\gamma)$.
Therefore $\gamma$ is homotopic to a smoothing $\check\gamma'$ of $\gamma'$. A theorem of Neumann-Coto%
\footnote{Neumann-Coto's proposition requires the multicurve to be made of primitive curves.
But in the proof, this is actually only needed to apply a theorem of Hass and Scott on monotonic simplification of curves, which is valid also for non-primitive multicurves, as was proved by de Graaf and Schrijver~\cite{gs-mcmcr-97} (or our Lemma~\ref{L:tighten-strip}).}%
~\cite[Proposition~2.2]{n-csgs-01} shows that a tightening of $\check\gamma'$ can be found among the smoothings of $\check\gamma'$, and thus of $\gamma'$.
Therefore without loss of generality we can take $\check\gamma'$ to be in minimal position.
Assume for contradiction that $x$ has not been smoothed in $\check\gamma'$.
Denote by $\kappa'$ the constituent curve of $\check\gamma'$ containing $x$,
and $\kappa$ the constituent curve of $\gamma$ homotopic to $\kappa'$.
Since $\kappa$ and $\kappa'$ are homotopic and both in minimal position, they can be transformed into each other via only $\arc33$ moves (see for example Hass and Scott~\cite[Theorem~2.1]{hs-scs-94}).
Since $\arc33$ moves preserve punctured monogons and $\gamma$ (and thus $\kappa$) has no punctured monogon, we have that $\kappa'$ contains no punctured monogon, a contradiction.
So the vertex $x$ has been smoothed in $\check\gamma'$.
While there are two possible ways to smooth $x$, one of them disconnects the multicurve which is impossible because $\gamma$ is a connected multicurve in minimal position and also homotopic to $\check\gamma'$.
The other smoothing, if performed on $\gamma'$, gives us $\gamma''$; this implies that $\check\gamma'$ can be viewed as a smoothing of $\gamma''$.
Therefore, $\gamma$ is homotopic to a smoothing of $\gamma''$.

Since $\gamma'$ and $\gamma$ are related via electrical moves, by Chang, Cossarini and Erickson~\cite[Lemma~3.1]{homoelectric}, $\gamma''$ (which is a smoothing of $\gamma'$) is related via electrical moves to a smoothing $\check\gamma$ of $\gamma$.
If $\check\gamma=\gamma$, then doing the moves in reverse, we have removed the need of a terminal-leaf contraction between $\gamma$ and $\gamma''$ and thus this contradicts the minimality of the number of terminal-leaf contractions used in our counter-example. Otherwise, $\gamma$ is homotopic to a smoothing of $\gamma''$, thus is in $\Route(\gamma'')$, and thus in $\Route(\check\gamma)$.
Thus $\gamma$ is homotopic to one of its strict smoothings, which contradicts that $\gamma$ is in minimal position.
\end{proof}

Now we are ready to prove the main result that any multicurve can be tightened using electrical moves and terminal-leaf contractions in polynomial time.

\begin{theorem}
\label{Th:homotopy-electric-TL}
Any multicurve $\gamma$ on a surface $\Sigma$ with $n$ vertices can be tightened using electrical moves and terminal-leaf contractions in $((g+b)\cdot n^5)$ time.
\end{theorem}

\begin{proof}
Let $\gamma$ be an arbitrary multicurve on $\Sigma$ with $n$ vertices.  Our recursive algorithm loops as follows.

By Lemma~\ref{L:tighten-strip} and Lemma~\ref{L:homotopy-electric}, one can tighten $\gamma$ into another multicurve $\gamma'$ using electrical moves in $O((g+b) \cdot n^4)$ time.
Now $\gamma'$ is in minimal position.
Without loss of generality, $\gamma'$ is connected, as disconnected components cannot interact with each other under electrical moves and terminal-leaf contractions.
For each vertex $x$ of $\gamma'$ we test whether $x$ is the tip of a punctured monogon using a linear-time homotopy test~\cite{lazarusrivaud,dehn}. If $\gamma'$ contains no punctured monogon, by Lemma~\ref{L:tightness}, $\gamma'$ is tight under electrical moves and terminal-leaf contractions.
Otherwise, if $\gamma'$ contains a punctured monogon, by Lemma~\ref{L:sing-monogons}, we can compute in $O(n^3)$ time a sequence of $O(n^3)$ $\arc33$ moves turning $\gamma'$ into a multicurve containing an empty punctured monogon.
Applying a terminal-leaf contraction on this empty punctured monogon, we obtain a new multicurve $\gamma''$ with $n-1$ vertices, and we go back to the start of the loop.

Since the number of vertices strictly decreases, the bottleneck of the algorithm is to tighten $\gamma$ using electrical moves and applying Lemma~\ref{L:sing-monogons}.
The total running time is $O((g+b) \cdot n^5)$.
\end{proof}


\fi


\paragraph{Acknowledgments.}  The authors would like thank Jeff Erickson and Francis Lazarus for helpful discussions, and Gelasio Salazar for his comments on earlier version of the paper.

\bibliographystyle{newuser}
\bibliography{tighter,bib/topology,bib/jeffe,bib/thesis}


\end{document}